\documentclass[final]{siamltex}
\usepackage{latexsym,amsfonts,amsmath,graphics,xypic,epsfig}

\newtheorem{remark}[theorem]{Remark}

\newcommand{\bsa}{\boldsymbol{a}}
\newcommand{\bsgamma}{\boldsymbol{\gamma}}
\newcommand{\bstau}{\boldsymbol{\tau}}

\newcommand{\bsdelta}{\boldsymbol{\delta}}
\newcommand{\bszeta}{\boldsymbol{\zeta}}
\newcommand{\bsc}{\boldsymbol{c}}
\newcommand{\bsk}{\boldsymbol{k}}
\newcommand{\bsl}{\boldsymbol{l}}
\newcommand{\bst}{\boldsymbol{t}}

\newcommand{\bsx}{\boldsymbol{x}}
\newcommand{\bsr}{\boldsymbol{r}}

\newcommand{\bsh}{\boldsymbol{h}}

\newcommand{\bsi}{\boldsymbol{i}}
\newcommand{\bsy}{\boldsymbol{y}}
\newcommand{\bssigma}{\boldsymbol{\sigma}}
\newcommand{\Dcal}{\mathcal{D}}
\newcommand{\Pcal}{\mathcal{P}}
\newcommand{\BB}{\mathcal{B}}
\newcommand{\wal}{{\rm wal}}
\newcommand{\sob}{{\rm sob}}
\newcommand{\soban}{{\rm sob, an}}
\newcommand{\icomp}{\mathtt{i}}
\newcommand{\bszero}{\boldsymbol{0}}
\newcommand{\bsone}{\boldsymbol{1}}
\newcommand{\rd}{\,\mathrm{d}}
\newcommand{\de}{\mathrm{e}}
\newcommand{\NN}{\mathbb{N}}
\newcommand{\LL}{\mathcal{L}}

\newcommand{\pp}{\mathfrak{p}}
\newcommand{\qq}{\mathfrak{q}}
\newcommand{\rr}{\mathfrak{r}}
\newcommand{\Vol}{{\rm Vol}}
\newcommand{\Landau}{\mathcal{O}}
\newcommand{\integer}{\mathbb{Z}}
\newcommand{\real}{\mathbb{R}}
\newcommand{\FF}{\mathbb{F}}
\newcommand{\cS}{\mathcal{S}}
\newcommand{\EE}{\mathcal{E}}

\begin{document}

\title{Walsh spaces containing smooth functions and quasi-Monte Carlo rules of arbitrary high order}

\author{Josef Dick\thanks{School of Mathematics and Statistics, University of New
South Wales, Sydney 2052, Australia. ({\tt josef.dick@unsw.edu.au})}}

\date{}

\maketitle

\begin{abstract}
We define a Walsh space which contains all functions whose partial
mixed derivatives up to order $\delta \ge 1$ exist and have finite
variation. In particular, for a suitable choice of parameters, this
implies that certain Sobolev spaces are contained in these Walsh
spaces. For this Walsh space we then show that quasi-Monte Carlo
rules based on digital $(t,\alpha,s)$-sequences achieve the optimal
rate of convergence of the worst-case error for numerical
integration. This rate of convergence is also optimal for the
subspace of smooth functions. Explicit constructions of digital
$(t,\alpha,s)$-sequences are given hence providing explicit
quasi-Monte Carlo rules which achieve the optimal rate of
convergence of the integration error for arbitrarily smooth
functions.
\end{abstract}



\pagestyle{myheadings}
\thispagestyle{plain}
\markboth{Josef DICK}{Explicit constructions of quasi-Monte Carlo rules achieving arbitrary high convergence}

\section{Introduction}
Quasi-Monte Carlo rules are quadrature rules which aim to
approximate an integral $\int_{[0,1]^s} f(\bsx)\rd\bsx$ by the
average of the $N$ function values $f(\bsx_n)$ at the quadrature
points $\bsx_0,\ldots,\bsx_{N-1} \in [0,1]^s$ (and hence are equal
weight quadrature rules). The dimension $s$ can be arbitrarily
large. The task here is to find ways of how to choose those
quadrature points in order to obtain a fast convergence of the
approximation to the integral. Explicit constructions of quadrature
points in arbitrary high dimensions are until now available for the
following two cases:
\begin{enumerate}
\item for sufficiently smooth periodic functions arbitrary high convergence
can be achieved using Kronecker sequences \cite[Theorem~5.3]{nie78} or
a modification of digital nets recently introduced in \cite{Dick05};
\item a convergence of $\Landau(N^{-1} (\log N)^{s-1})$ can be achieved
for functions of bounded variation (in this case the functions are
not required to be periodic).
\end{enumerate}
For non-periodic functions no explicit constructions have been
established which can fully exploit the smoothness of the integrand.
This paper provides a complete solution to this problem.

Among other things we show that an explicit construction of suitable
point sets and sequences can be obtained in the following way: let
$d \ge 1$ be an integer and let $\bsx_0,\bsx_1,\ldots \in
[0,1)^{ds}$ be the points of a digital $(t,m,ds)$-net or digital
$(t,ds)$-sequence over a finite field $\FF_q$ in dimension $d s$
(see \cite{niesiam} for the definition of digital nets and sequences
and see for example \cite{faure,niesiam,nierev,NX,sob67} for
explicit constructions of suitable digital nets and sequences). Let
$\bsx_n = (x_{n,1},\ldots, x_{n,ds})$ with $x_{n,j} = x_{n,j,1}
q^{-1} + x_{n,j,2} q^{-2} + \cdots$ and $x_{n,j,i} \in \{0,\ldots,
q-1\}$ (i.e. $x_{n,j,i}$ are the digits in the base $q$
representation of $x_{n,j}$). Then for $n \ge 0$ we define $\bsy_n =
(y_{n,1},\ldots, y_{n,s})$ with
$$y_{n,j} = \sum_{i=1}^\infty \sum_{k=1}^d x_{n,(j-1) d + k,i} q^{-k
- (i-1) d} \quad\mbox{for } j = 1,\ldots, s.$$ (Note that the
addition here is carried out in $\real$ and that the sum over $i$
above is often finite as $x_{n,j,i} = 0$ for $i$ large enough.) We
point out here that the quality of the point set or sequence is
directly related to the $t$-value of the underlying $(t,m,ds)$-net
or $(t,ds)$-sequence, see Theorem~\ref{th_talphabeta} and
Theorem~\ref{th_talphabetaseq}.

Corollary~\ref{cor_errorbounddignet} now shows that quasi-Monte
Carlo rules using the points $\bsy_0,\ldots, \bsy_{N-1}$ (with $N =
q^m$ for some $m \ge 1$) achieve the optimal rate of convergence of
the integration error of $\Landau(N^{-\vartheta} (\log N)^{\vartheta
s})$ for functions which have partial mixed derivatives up to order
$\vartheta$ which are square integrable as long as $1 \le \vartheta
\le d$ (Corollary~\ref{cor_errorbounddignet} is actually more
general). If $\vartheta > d$ no improvement of the convergence rate
is obtained compared to functions with smoothness $\vartheta = d$,
i.e. we obtain a convergence of $\Landau(N^{-d} (\log N)^{d s})$.
Similar, but less general results for periodic functions compared to
those in this paper have been shown in \cite{Dick05} by a different
proof method. (The construction above is an example of a
construction method which can be used. In Section~\ref{sectdignets}
we outline the general algebraical properties required for the
construction of suitable point sets.)

The quasi-Monte Carlo algorithm based on digital nets and sequences
proposed here has also some further useful properties. For example
our results also hold if one randomizes the point set by, say, a
random digital shift (see for example \cite{DP05,DP05b,matou}).
(This follows easily because the worst-case error (see
Section~\ref{sect_int}) is invariant with respect to digital shifts
in the Walsh space and hence we obtain the same upper bounds for
randomized digital nets and sequences.) In summary the quadrature
rules have the following properties:
\begin{itemize}
\item The quadrature rules introduced in this paper are equal weight quadrature rules which achieve the optimal rate of convergence up to some $\log N$ factors and the result holds for deterministic and randomly digitally shifted quadrature rules.
\item The construction of the underlying point set is explicit and suitable point sets are available in arbitrary high dimensions and arbitrary high number of points.
\item The quadrature rules automatically adjust themselves to the optimal rate of convergence $\Landau(N^{-\vartheta} (\log N)^{s\vartheta})$ as long as $1 \le \vartheta \le d$.
\item The underlying point set is extensible in the dimension as well as in the number of points, i.e., one can always add some coordinates or points to an existing point set such that the quality of the point set is preserved.
\end{itemize}

In the following we lay out some of the underlying principles used
in this work which stem from the behaviour of the Walsh coefficients
of smooth functions. Walsh functions are piecewise constant wavelets
which form an orthonormal set of $\LL_2([0,1]^s)$. In their simplest
form, for a non-negative integer $k$ with base $2$ representation $k
= \kappa_0 + \cdots + \kappa_{m-1} 2^{m-1}$ and an $x \in [0,1)$
with base $2$ representation $x = x_1 2^{-1} + x_2 2^{-2} + \cdots$,
the $k$-th Walsh function in base $2$ is given by $$\wal_k(x) =
(-1)^{\kappa_0 x_1 + \cdots + \kappa_{m-1} x_m}.$$ (Later on we will
use the more general definition of Walsh functions over groups.)

The behaviour of the Fourier coefficients of smooth periodic
functions is well known, i.e. the smoother the function the faster
the Fourier coefficients go to zero (see for example \cite{zyg}). An
analogous result for Walsh functions has, to the best of the authors
knowledge, not been known until now (see Fine~\cite{Fine} who, for
example, shows that the only absolute continuous functions whose
$k$-th Walsh coefficients decay faster than $1/k$ are constant
functions). This will be established here and subsequently be
exploited to obtain quasi-Monte Carlo rules with arbitrary high
order of convergence.

To give a glimpse of how the Walsh coefficients of smooth functions
behave, consider for example the Walsh series for $1/2-x$:
$$1/2-x = \sum_{k=0}^\infty c_k \wal_k(x) = \sum_{a=0}^\infty 2^{-a-2} \wal_{2^a}(x).$$ Although the
function is infinitely smooth, in general the decay of the Walsh
coefficient is only of order $1/k$. But note that most of the Walsh
coefficients are actually $0$. For example when we consider
$(1/2-x)^2$, then typically we would have that the Walsh coefficient
of $k = 2^a$ is of order $2^{-a}$, the Walsh coefficient of $k =
2^{a_1} + 2^{a_2}$ ($a_1 > a_2$) is of order $2^{-a_1-a_2}$ and for
$k = 2^{a_1} + 2^{a_2} + \cdots + 2^{a_v}$ with $a_1 > \cdots > a_v$
and $v > 2$ the $k$-th Walsh coefficient would be $0$.  By
considering $(1/2-x)^3, (1/2-x)^4, \ldots$, or more generally
polynomials, one can now realize that the speed of convergence of
the Walsh coefficients depends on how many non-zero digits $k$ has.
This is the basic feature which we will relate to the speed of
convergence of the Walsh coefficients for smooth functions.

Subsequently we will explicitly state and use the behaviour of the
Walsh coefficients of smooth functions. In general the Walsh
functions depend on the base $q$ digit expansion of the wavenumber
$k$ and also of the point $x$ where the Walsh function is to be
evaluated. Hence, maybe not surprisingly, the value of the $k$-th
Walsh coefficients of smooth functions also depend on the $q$-adic
expansion of $k$. We show that the Walsh
space~$\EE_{s,q,\vartheta,\bsgamma}$ introduced in
Section~\ref{sectwalshspace} contains all functions whose partial
mixed derivatives up to order $\delta < \vartheta$ exist and have
finite variation, where $\vartheta$ is a parameter restricting the
behaviour of the Walsh coefficients of the function space
$\EE_{s,q,\vartheta,\bsgamma}$. (We use a similar, though much more
general, technique as Fine~\cite{Fine} used for showing that the
Walsh coefficients of a differentiable function cannot decay faster
than $1/k$.)

The concept of digital $(t,\alpha,\beta,m,s)$-nets and digital
$(t,\alpha,\beta,s)$-sequences (see Section~\ref{sectdignets} and
also \cite{Dick05} for a similar concept) is now designed to yield
point sets which work well for the Walsh
space~$\EE_{s,q,\vartheta,\bsgamma}$, just in the same way as the
digital nets and sequences from \cite{faure,nie86,niesiam,NX,sob67}
are designed to work well for the spaces for example considered in
\cite{DP05,HHY} (or as lattice rules are designed to work well for
periodic Korobov spaces). Here the power of the result that the
Walsh space~$\EE_{s,q,\vartheta,\bsgamma}$ contains smooth functions
comes into play: it follows that we can fully exploit the smoothness
of an integrand using digital $(t,\alpha,\alpha,m,s)$-nets or
digital $(t,\alpha,\alpha,s)$-sequences. As the construction of the
points $\bsy_0,\bsy_1,\ldots$ introduced at the beginning yields
explicit examples of digital $(t,\alpha,\alpha,m,s)$-nets or digital
$(t,\alpha,\alpha,s)$-sequences as shown in
Section~\ref{sectdignets} we therefore obtain explicit constructions
of quasi-Monte Carlo rules which can achieve the optimal order of
convergence for arbitrary smooth functions.

In the next section we introduce Walsh functions over groups and
state some of their essential properties.

\section{Walsh functions over groups}

In this section we give the definition of Walsh functions over
groups and present some essential properties. Walsh functions in
base $2$ were first introduced by Walsh~\cite{walsh}, though a
similar but non-complete set of functions has already been studied
by Rademacher~\cite{rade}. Further important results were obtained
in \cite{Fine}. We follow \cite{PDP} in our presentation.

\subsection{Definition of Walsh functions over groups}

 An essential tool for the investigation of digital nets are Walsh functions.
A very general definition, corresponding to the most general
construction of digital nets over finite rings, was given in
\cite{LNS}. There, Walsh functions over a finite abelian group $G$,
using some bijection $\varphi$, were defined.   Here we restrict
ourselves to the additive groups  of the finite fields $\FF_{p^r}$,
$p$ prime and $r \ge 1$. We restate the definitions for this special
case here for the sake of convenience. In the following let $\NN$
denote the set of positive integers and $\NN_0$ the set of
non-negative integers.

\begin{definition}[Walsh functions]\label{defwalsh}\rm
Let $q=p^r$, $p$ prime, $r\in \NN$ and let $\FF_q$ be the finite
field with $q$ elements. Let
$\integer_q=\{0,1,\ldots,q-1\}\subset\integer$ and let
$\varphi:\integer_q \longrightarrow \FF_q$ be a bijection such that
$\varphi(0)=0$, the neutral element of addition in $\FF_q$. Moreover
denote by $\psi$ the
 canonical isomorphism (described below) of additive groups $\psi:\FF_q \longrightarrow
\integer_p^r$ and define $\eta:=\psi \circ \varphi$. For $1 \le i
\le r$ denote by $\pi_i$ the projection
$\pi_i:\integer_p^r\longrightarrow \integer_p$,
$\pi_i(x_1,\ldots,x_r)=x_i$.
\[ \xymatrix{ \integer_{q} \ar[r]^{\varphi} \ar[dr]_{\eta} & \FF_{q}
\ar[d]^{\psi} \\ & \integer_p^{r} \ar[r]^{\pi_i} & \integer_p } \]
Let now $k \in \NN_0$ with base $q$ representation
$k=\kappa_0+\kappa_1 q+\cdots +\kappa_{m-1} q^{m-1}$ where $\kappa_l
\in \integer_q$ and let $x \in [0,1)$ with base $q$ representation
$x=x_1/q+x_2/q^2+\cdots$ (unique in the sense that infinitely many
$x_l$ must be different from $q-1$). Then the $k$-th Walsh function
over the additive group of the finite field $\FF_q$ with respect to
the bijection $\varphi$ is defined by
$$_{\FF_q,\varphi}\wal_k(x) = \exp\left(\frac{2\pi \icomp}{p} \sum_{l=0}^{m-1}\sum_{i=1}^r (\pi_i \circ \eta)(\kappa_l)(\pi_i \circ
\eta)(x_{l+1})\right).$$ For convenience we will in the rest of the
paper omit the subscript and simply write ${\wal}_k$ if there is no
ambiguity.

Multivariate Walsh functions are defined by multiplication of the univariate
components, i.e., for $s >1$, $\bsx=(x_1,\ldots,x_s)\in [0,1)^s$ and
$\bsk=(k_1,\ldots,k_s) \in\NN^s_0$, we set \[
\wal_{\bsk}(\bsx) = \prod_{j=1}^s \wal_{k_j}(x_j) .\]
\end{definition}

We now briefly describe the canonical isomorphism. Let
$\FF_q=\integer_p[\theta]$, such that
$\{1,\theta,\ldots,\theta^{r-1}\}$ is a basis of $\FF_q$ over
$\integer_p$ as a vector space. Then the isomorphism $\psi$ between
$\FF_q$ and $\integer_p^r$ shall be given by
\[ \psi(x)=(x_1,\ldots,x_r)^\top,\text{ for }x=\sum_{i=1}^r x_i\theta^{i-1},
   x_i \in \integer_p.
\] For more information on the Walsh functions defined above see \cite{PDP}.

We summarize some important properties of Walsh functions over the
additive group of a finite field which will be used throughout the
paper. The proofs of the subsequent results can be found e.g.\ in
\cite{larpir,pirsic} (see also \cite{chrest}). In the following we call $x \in [0,1)$ a
$q$-adic rational if $x$ can be represented by a finite base $q$
expansion.

\begin{proposition} \label{walshprops}
Let $p$, $q$, $\FF_q$ and $\varphi$ be as in Definition
\ref{defwalsh}. For $x,y$ with $q$-adic representations
$x=\sum_{i=w}^{\infty}{x_i}{q^{-i}}$ and
$y=\sum_{i=w}^{\infty}{y_i}{q^{-i}}$, $w\in\integer$ (taking $w$
negative, hence the following operations are also defined for
integers), define $x \oplus_{\varphi}
y:=\sum_{i=w}^{\infty}{z_i}{q^{-i}}$ where
$z_i:=\varphi^{-1}(\varphi(x_i)+\varphi(y_i))$ and
$\ominus_{\varphi}x:=\sum_{i=w}^{\infty}{v_i}{q^{-i}}$ where
$v_i:=\varphi^{-1}(-\varphi(x_i))$. Further we set $x\ominus_\varphi
y:= x\oplus_\varphi(\ominus_\varphi y)$. For vectors $\bsx, \bsy$ we
define the operations component-wise. Then we have:
\begin{enumerate}
\item
For all $k,l \in \NN_0$ and all $x,y \in [0,1)$, with the
restriction that if $x, y$ are not $q$-adic rationals then $x
\oplus_{\varphi} y$ is not allowed to be a $q$-adic rational, we
have
$$\wal_k(x) \cdot \wal_l(x) = \wal_{k \oplus_{\varphi} l}(x), \; \; \;
\wal_k(x) \cdot \wal_k(y) = \wal_k(x \oplus_{\varphi} y)$$ and, with
the restriction that if $x, y$ are not $q$-adic rationals then $x
\ominus_{\varphi} y$ is not allowed to be a $q$-adic rational,
$$\wal_k(x) \cdot \overline{\wal_l(x)} = \wal_{k \ominus_{\varphi}
l}(x), \; \; \; \wal_{k}(x) \cdot \overline{\wal_{k}(y)} =
\wal_k(x\ominus_{\varphi} y). $$
\item We have
 $$\sum_{k=0}^{q-1} \wal_{l}(k/q) = \begin{cases}0&\text{if } l\neq 0,\\
                  q&\text{if } l=0. \end{cases} $$
\item We have $$\int_0^1\wal_0(x)\rd x=1 \;\;\; \text{ and } \;\;\;\int_0^1\wal_k(x)\rd x=0 \text{ if } k>0.$$
\item For all $\bsk, \bsl \in \NN_0^s$ we have the following orthogonality properties: $$\int_{[0,1)^s} \wal_{\bsk}(\bsx) \overline{\wal_{\bsl}(\bsx)}\rd \bsx =
\begin{cases}
1 & \text{ if } \bsk=\bsl, \\
0 & \text{ otherwise}.
\end{cases} $$
\item For any $f \in \LL_2([0,1)^s)$ and any $\bssigma \in [0,1)^s$ we have
$$\int_{[0,1)^s}f(\bsx)\rd \bsx = \int_{[0,1)^s}f(\bsx \oplus_{\varphi}
\bssigma)\rd \bsx.$$ \item For any integer $s \ge 1$ the system $\{\wal_{\bsk}:
\bsk \in \NN_0^s\}$ is a complete orthonormal system in $\LL_2([0,1)^s)$.
\end{enumerate}
\end{proposition}

\begin{remark}\rm
The restrictions in item $1.$ was added to exclude cases like: $x =
(0.010101\ldots)_2$, $y = (0.0010101\ldots)_2$ and $x \oplus y =
(0.1)_2$, for which the result is of course not true. On the other
hand, the result holds for $x \oplus y = (0.0111111\ldots)_2$.
\end{remark}

Throughout the paper we will use a fixed bijection $\varphi$ and a
fixed finite field $\FF_q$ is used for Walsh functions and
$\oplus_{\varphi}$ and $\ominus_{\varphi}$. Hence we will often
write $\oplus$ and $\ominus$ instead of $\oplus_{\varphi}$ and
$\ominus_{\varphi}$.

In the following section we will deal with Walsh series and Walsh
coefficients, which we briefly describe in the following: functions
$f \in \LL_2([0,1)^s)$ have an associated Walsh series
$$f(\bsx) \sim \sum_{\bsk \in \NN_0^s} \hat{f}(\bsk)
\wal_{\bsk}(\bsx),$$ where the Walsh coefficients $\hat{f}(\bsk)$
are given by $$\hat{f}(\bsk) = \int_{[0,1)^s} f(\bsx)
\wal_{\bsk}(\bsx) \rd \bsx.$$ For smooth functions the Walsh series
converges to the function, which is shown in
Section~\ref{subsect_convergence}.

\section{Walsh spaces containing smooth functions}\label{sectwalshspace}

In the following we investigate how the Walsh coefficients of smooth
functions decay and subsequently we use this to define function
classes based on Walsh functions which contain smooth functions. But
first we introduce a suitable variation.

\subsection{A generalized weighted Hardy and Krause variation}

In the following we generalize the Hardy and Krause variation which
suits our purposes later on.

\subsubsection{H\"older condition}

A function $f:[0,1) \rightarrow \real$ satisfies a H\"older
condition with coefficient $0 < \lambda \le 1$ if there is a
constant $C_f >0$ such that $$|f(x) - f(y)| \le C_f  |x-y|^\lambda
\quad \mbox{for all } x,y \in [0,1).$$ The right hand side of the
above inequality forms a metric on $[0,1)$. When one considers the
higher dimensional domain $[0,1)^s$ then $|x-y|$ is changed to some
other metric on $[0,1)^s$. Here we consider tensor product spaces
and we generalize the H\"older condition to higher dimensions in a
way which is suitable for tensor product spaces in our context.
Consider for example the function $f(\bsx) = \prod_{j=1}^s
f_j(x_j)$, where $\bsx = (x_1,\ldots, x_s)$ and each
$f_j:[0,1)\rightarrow \real$ satisfies a H\"older condition with
coefficient $0<\lambda \le 1$. Then it follows that for all
$\emptyset \neq u \subseteq \cS := \{1,\ldots, s\}$ we have
\begin{equation}\label{eq_prodholder}
\prod_{j \in u} |f_j(x_j) - f_j(y_j)| \le \prod_{j\in u} C_{f_j} \prod_{j \in u} |x_j - y_j|^\lambda
\end{equation}
for all $x_j,y_j \in [0,1)$ with $j \in u$.
But here $\prod_{j=1}^s |x_j - y_j|$ is not a metric on $[0,1)^s$.

Note that we have
\begin{equation}\label{eq_sumprod}
\prod_{j \in u} |f_j(x_j) - f_j(y_j)| = \left|\sum_{v \subseteq u} (-1)^{|v|-|u|} \prod_{j\in v}  f_j(x_j) \prod_{j\in u\setminus v} f_j(y_j)\right|,
\end{equation}
which can be described in words in the following way: for given
$\emptyset \neq u \subseteq \cS$ let $x_j, y_j \in [0,1)$ with $x_j
\neq y_j$ for all $j \in u$; consider the box $J$ with vertices
$\{(a_j)_{j\in u}: a_j = x_j \mbox{ or } a_j = y_j \mbox{ for } j
\in u\}$. Then (\ref{eq_sumprod}) is the alternating sum of the
function $\prod_{j\in u} f_j$ at the vertices of $J$ where adjacent
vertices have opposite signs. This sum can also be defined for
functions on $[0,1)^s$ which are not of product form.

Indeed, let for a subinterval $J = \prod_{j=1}^s [x_j, y_j)$ with $0 \le x_j < y_j \le 1$ and a function $f:[0,1)^s \rightarrow \real$ the function $\Delta(f,J)$ denote the alternating sum of $f$ at the vertices of $J$ where adjacent vertices have opposite signs. (Hence for $f = \prod_{j=1}^s f_j$ we have $\Delta(f,J) = \prod_{j=1}^s (f_j(x_j) - f_j(y_j))$.)

\subsubsection{Generalized Vitali variation}

Let $\pp \ge 1$. Then we define the generalized variation in the
sense of Vitali with coefficient $0 < \lambda \le 1$  by
\begin{equation}\label{fracVitalivar} V^{(s)}_{\lambda, \pp}(f) = \sup_{{\Pcal}} \left(\sum_{J \in
\Pcal} \Vol(J)
\left|\frac{\Delta(f,J)}{\Vol(J)^{\lambda}}\right|^{\pp}\right)^{1/\pp},\end{equation}
where the supremum is extended over all partitions $\Pcal$ of
$[0,1]^s$ into subintervals and $\Vol(J)$ denotes the volume of the
subinterval $J$.

Note that for $\lambda = 1$ and $\pp = 1$ one obtains the usual
definition of the Vitali variation, see for example \cite{niesiam}.
If we take $\pp = \infty$, then we obtain a condition of the form
(\ref{eq_prodholder}) where $u = \cS$ and where we can take the
constant $\prod_{j=1}^s C_{f_j} = V^{(s)}_{\lambda,\infty}(f)$. For
$s = 1$ and $\pp = \infty$ we obtain a H\"older condition with
coefficient $0 < \lambda \le 1$. In this sense we can view
(\ref{fracVitalivar}) as a fractional Vitali variation of order
$\lambda$.

For $\lambda = 1$ and if the partial derivatives of $f$ are continuous on $[0,1]^s$ we also have the formula
\begin{equation}\label{eq_formelV}
V_{1,\pp}^{(s)}(f) = \left(\int_{[0,1]^s} \left|\frac{\partial^s f}{\partial x_1\cdots \partial x_s} \right|^{\pp} \rd \bsx\right)^{1/\pp},
\end{equation}
for all $\pp \ge 1$. Indeed we have $$|\Delta(f,J)| = \left|\int_J\frac{\partial^s f}{\partial x_1\cdots \partial x_s}(\bsx) \rd \bsx \right| = \Vol(J) \left|\frac{\partial^s f}{\partial x_1\cdots \partial x_s}(\bszeta_J)\right|$$ for some $\bszeta_J \in \overline{J}$, which follows by applying the mean value theorem to the inequality
\begin{equation*}
\min_{\bsx \in \overline{J}} \left|\frac{\partial^s f}{\partial x_1\cdots \partial x_s}(\bsx) \right| \le \Vol(J)^{-1}  \left|\int_J\frac{\partial^s f}{\partial x_1\cdots \partial x_s}(\bsx) \rd \bsx \right|  \le   \max_{\bsx\in \overline{J}} \left|\frac{\partial^s f}{\partial x_1\cdots \partial x_s}(\bsx)\right|.
\end{equation*}
Therefore we have $$ \sum_{J \in \Pcal} \Vol(J)
\left|\frac{\Delta(f,J)}{\Vol(J)}\right|^{\pp} = \sum_{J\in\Pcal}
\Vol(J) \left|\frac{\partial^s f}{\partial x_1\cdots \partial
x_s}(\bszeta_J)\right|^\pp,$$ which is just a Riemann sum for the
integral $\int_{[0,1]^s} \left|\frac{\partial^s f}{\partial
x_1\cdots \partial x_s} \right|^{\pp} \rd \bsx$ and thus the
equality follows.

Using H\"older's inequality and the fact that $\left(\sum_{J \in
\Pcal} (\Vol(J)^{1-1/\pp})^{\pp/(\pp-1)}\right)^{1-1/\pp} =
\left(\sum_{J\in \Pcal} \Vol(J)\right)^{1-1/\pp} = 1$ it follows
that $$V_{\lambda,1}^{(s)}(f) \le V_{\lambda,\pp}^{(s)}(f) \quad
\mbox{for all } \pp \ge 1.$$

\subsubsection{Generalized Hardy and Krause variation}

Until now we did not take projections to lower dimensional faces
into account (in (\ref{eq_prodholder}) we did take projections into
account as we considered all $\emptyset \neq u \subseteq\cS$).

For $\emptyset \neq u \subseteq\cS$, let $V_{\lambda,
\pp}^{(|u|)}(f_u;u)$ be the generalized Vitali variation with
coefficient $0 < \lambda \le 1$ of the $|u|$-dimensional function
$f_u(\bsx_u) = \int_{[0,1)^{s-|u|}} f(\bsx) \rd \bsx_{\cS\setminus
u}$. For $u  = \emptyset$ we have $f_\emptyset = \int_{[0,1)^{s}}
f(\bsx) \rd \bsx_{\cS}$ and we define $V_{\lambda,
\pp}^{(|\emptyset|)}(f_\emptyset;\emptyset) = |f_\emptyset|$. Let
$\qq \ge 1$, then
\begin{equation}\label{eq_varhk}
V_{\lambda, \pp, \qq}(f) = \left(\sum_{u \subseteq\cS} \left(V^{(|u|)}_{\lambda, \pp}(f_u;u) \right)^{\qq}\right)^{1/\qq}
\end{equation}
is called the generalized Hardy and Krause variation of $f$ on
$[0,1]^s$.

For $\lambda = \pp = \qq = 1$ one obtains an unanchored version of
the usual definition of the Hardy and Krause variation, see
\cite{niesiam}. A function $f$ for which $V_{\lambda, \pp, \qq}(f) <
\infty$ is said to be of finite variation with coefficient
$\lambda$. (We remark that in some cases it might be appropriate to
leave out the term corresponding to $u = \emptyset$ in
(\ref{eq_varhk}), but here this term will be needed later on and
hence we include it already in the definition of the variation.)

\subsubsection{Generalized weighted Hardy and Krause variation}

As first suggested in \cite{SW98} (see also \cite{DSWW1}) different
coordinates might have different importance, hence we can also
define a weighted variation. In the spirit of the weighted Sobolev
spaces in \cite{DSWW1}, let $\bsgamma = (\gamma_u)_{u \subset \NN}$
be an indexed set of non-negative real numbers. Then we define the
weighted variation $V_{\lambda,\pp,\qq,\bsgamma}(f)$ of $f$ with
coefficient $0 < \lambda \le 1$ by \begin{equation*} V_{\lambda,\pp,
\qq,\bsgamma}(f) = \left(\sum_{u \subseteq\cS} \gamma_u^{-1}
\left(V^{(|u|)}_{\lambda, \pp}(f_u;u) \right)^{\qq}\right)^{1/\qq}.
\end{equation*}

Note that for $\lambda = 1$ and $\pp = \qq = 2$ the weighted
variation $V_{\lambda,\pp,\qq,\bsgamma}(f)$ coincides with the norm
in a weighted unanchored Sobolev space for any function in this
Sobolev space, i.e, we have the identity $V_{1,2,2,\bsgamma}(f)
=\|f\|_{\sob}$, where
\begin{equation*}
\|f\|_{\sob} = \left(\sum_{u\subseteq \cS} \gamma_u^{-1} \int_{[0,1)^{|u|}} \left|\int_{[0,1)^{s-|u|}} \frac{\partial^{|u|} f(\bsx)}{\partial \bsx_u} \rd\bsx_{\cS\setminus u} \right|^2 \rd \bsx_u\right)^{1/2}
\end{equation*}
denotes the norm in the weighted Sobolev space (see \cite{DSWW1} for more information on this Sobolev space).

\subsection{The decay of the Walsh coefficients of smooth functions}

We are now ready to show how the Walsh coefficients of smooth
functions decay. This behaviour is essentially captured in
Definition~\ref{def_qadicnonincreasing} below.  But before we get
there we need several lemmas to prove the result. The following
lemma is needed to show how the Walsh coefficients of functions with
bounded variation decay. A simpler version of it was shown in
\cite[Lemma~4]{pirsic}.
\begin{lemma}\label{lem_pirs}
Let $f \in \LL_1([0,1)^s)$ and let $\bsk = (k_1,\ldots, k_s) \in \NN^s$ with $k_j = \kappa_j q^{a_j-1} + k'_j$ where $a_j \in \NN$, $\kappa_j \in \{1,\ldots, q-1\}$, $0 \le k'_j < q^{a_j-1}$ and let $0 \le c_j < q^{a_j-1}$ for $j = 1,\ldots, s$. Then
\begin{eqnarray*}
\left|\int_{\prod_{j=1}^s [c_j q^{-a_j+1}, (c_j+1)q^{-a_j+1})}
f(\bsx) \; \overline{\wal_{\bsk}(\bsx)} \rd \bsx \right| & \le &
q^{-\sum_{j=1}^s (a_j -1)} \sup_{J} |\Delta(f,J)|,
\end{eqnarray*}
where the supremum is taken over all boxes of the form $$J =
\prod_{j=1}^s [d_j, e_j) \subseteq \prod_{j=1}^s [c_j q^{-a_j+1},
(c_j+1) q^{-a_j+1})$$ with $q^{a_j}|e_j-d_j| \in \{1,\ldots, q-1\}$.
\end{lemma}
\begin{proof}
We have $\overline{\wal_{k_j}} = \overline{\wal_{\kappa_j q^{a_j-1}}} \; \overline{\wal_{k'_j}}$ and the function $\overline{\wal_{k'_j}}$ is constant on each subinterval $[c_j q^{-a_j+1}, (c_j+1) q^{-a_j+1})$. Hence we have
\begin{eqnarray*}
\lefteqn{\left|\int_{\prod_{j=1}^s [c_j q^{-a_j+1}, (c_j+1)q^{-a_j+1})} f(\bsx) \; \overline{\wal_{\bsk}(\bsx)} \rd \bsx \right| } \\ & = & \left|\int_{\prod_{j=1}^s [c_j q^{-a_j+1}, (c_j+1)q^{-a_j+1})} f(\bsx) \; \prod_{j=1}^s \overline{\wal_{\kappa_j q^{a_j-1}}(x_j)} \rd \bsx \right|.
\end{eqnarray*}
Note that the function $\overline{\wal_{\kappa_j q^{a_j-1}}}$ is constant on each of the subintervals $[r_j q^{-a_j}, (r_j +1) q^{-a_j})$ for $r_j = 0,\ldots, q^{a_j}-1$ for $j = 1,\ldots, s$. Without loss of generality we may assume that $c_j = 0$, for all other $c_j$ the result follows by the same arguments. Thus we have
\begin{eqnarray*}
\lefteqn{\int_{\prod_{j=1}^s [0, q^{-a_j+1})} f(\bsx) \; \prod_{j=1}^s \overline{\wal_{\kappa_j q^{a_j-1}}(x_j)} \rd \bsx } \\ & = & \sum_{r_1,\ldots, r_s = 0}^{q-1} \prod_{j=1}^s \overline{\wal_{\kappa_j}(r_j/q)} \int_{\prod_{j=1}^s [r_j q^{-a_j}, (r_j +1) q^{-a_j})} f(\bsx) \rd \bsx.
\end{eqnarray*}
Let now $a_{(r_1,\ldots, r_s)} = \int_{\prod_{j=1}^s [r_j q^{-a_j},
(r_j+1) q^{-a_j})} f(\bsx) \rd \bsx$ and for a given $0 \le
r_1,\ldots, r_s < q$ let \begin{equation}\label{def_A} A(r_1,\ldots,
r_s) = q^{-s} \sum_{t_1,\ldots, t_s = 0}^{q-1} \sum_{\emptyset \neq
u \subseteq\{1,\ldots, s\}} (-1)^{|u|} a_{(\bst_u,\bsr_{\{1,\ldots,
s\}\setminus u})},\end{equation} where $(\bst_u,\bsr_{\{1,\ldots,
s\}\setminus u})$ denotes the vector obtained by setting the $j$-th
coordinate to $t_j$ if $j \in u$ and $r_j$ if $j \notin u$. Further
let
$$B(r_1,\ldots, r_s) = q^{-s} \sum_{t_1,\ldots, t_s = 0}^{q-1}
\sum_{u \subseteq\{1,\ldots, s\}} (-1)^{|u|}
a_{(\bst_u,\bsr_{\{1,\ldots, s\}\setminus u})}.$$ Then we have
\begin{eqnarray*}
\lefteqn{\sum_{r_1,\ldots, r_s =0}^{q-1} \prod_{j=1}^s
\overline{\wal_{\kappa_j}(r_j/q)} a_{(r_1,\ldots, r_s)} } \hspace{2cm} \\
& = & -\sum_{r_1,\ldots, r_s =0}^{q-1} \prod_{j=1}^s
\overline{\wal_{\kappa_j}(r_j/q)} A(r_1,\ldots, r_s) \\ && +
\sum_{r_1,\ldots, r_s =0}^{q-1} \prod_{j=1}^s
\overline{\wal_{\kappa_j}(r_j/q)} (a_{(r_1,\ldots, r_s)} +
A(r_1,\ldots, r_s)).
\end{eqnarray*}
Since $\sum_{r=0}^{q-1} \overline{\wal_{\kappa}(r/q)}  = 0$ and
$A(r_1,\ldots, r_s)$ is a sum where each summand does not depend on
at least one $r_j$, i.e. the case $u = \emptyset$ is excluded in
(\ref{def_A}), it follows that the first sum on the right hand side
above is zero. Further we have $a_{(r_1,\ldots, r_s)} +
A(r_1,\ldots, r_s) = B(r_1,\ldots, r_s)$ and thus
\begin{equation*}
\left|\sum_{r_1,\ldots, r_s =0}^{q-1} \prod_{j=1}^s \overline{\wal_{\kappa_j}(r_j/q)} a_{(r_1,\ldots, r_s)} \right| \le \sum_{r_1,\ldots, r_s =0}^{q-1} \left|B(r_1,\ldots, r_s) \right|.
\end{equation*}
We have $$|B(r_1,\ldots, r_s)| \le \max_{\bst \in \{0,\ldots, q-1\}^s} \left|\sum_{u \subseteq\{1,\ldots, s\}} (-1)^{|u|} a_{(\bst_u,\bsr_{\{1,\ldots, s\}\setminus u})} \right|.$$

Therefore we have
\begin{eqnarray*}
\lefteqn{ \left|\int_{\prod_{j=1}^s [0,q^{-a_j+1})} f(\bsx) \; \prod_{j=1}^s \overline{\wal_{\kappa_j q^{a_j-1}}(x_j)} \rd \bsx \right| } \\ & \le & \sum_{r_1,\ldots, r_s = 0}^{q-1} \max_{\bst \in \{0,\ldots, q-1\}^s} \left|\sum_{u \subseteq\{1,\ldots, s\}} (-1)^{|u|} a_{(\bst_u,\bsr_{\{1,\ldots, s\}\setminus u})} \right| \\ & \le &  q^s \max_{\bsr, \bst \in \{0,\ldots, q-1\}^s}  \left|\sum_{u \subseteq\{1,\ldots, s\}} (-1)^{|u|} a_{(\bst_u,\bsr_{\{1,\ldots, s\}\setminus u})} \right|.
\end{eqnarray*}
Note that if in the above maximum there is a $j$ such that $r_j = t_j$ then it follows that
$$\left|\sum_{u \subseteq\{1,\ldots, s\}} (-1)^{|u|} a_{(\bst_u,\bsr_{\{1,\ldots, s\}\setminus u})} \right| = 0.$$ Hence we may in the following assume without loss of generality that the maximum in the last line of the inequality above is taken on for $\bsr = (r_1,\ldots, r_s)$ and $\bst = (t_1,\ldots, t_s)$ which satisfy $r_j \neq t_j$ for $j = 1,\ldots, s$.

We have
\begin{eqnarray*}
\lefteqn{ \left|\sum_{u \subseteq\{1,\ldots, s\}} (-1)^{|u|} a_{(\bst_u,\bsr_{\{1,\ldots, s\}\setminus u})} \right| } \\ & = & \left|\int_{\prod_{j=1}^s [r_j q^{-a_j}, (r_j+1) q^{-a_j})} \sum_{u\subseteq\{1,\ldots, s\}} (-1)^{|u|} f(\bsx + \bsy_{u}) \rd \bsx \right|,
\end{eqnarray*}
where $\bsy_u = (y_1,\ldots, y_s)$ with $y_j = 0$ for $j \notin u$
and $y_j = (t_j-r_j) q^{-a_j}$ for $j \in u$. We can write
$$\sum_{u\subseteq\{1,\ldots, s\}} (-1)^{|u|} f(\bsx + \bsy_{u}) =
\Delta(f,J_{\bsx}),$$ where $J_{\bsx} = \prod_{j=1}^s [\min(x_j, x_j
+ (t_j-r_j) q^{-a_j}), \max(x_j, x_j + (t_j-r_j) q^{-a_j}))$.
Therefore it follows that
\begin{eqnarray*}
\left|\sum_{u \subseteq\{1,\ldots, s\}} (-1)^{|u|} a_{(\bst_u,\bsr_{\{1,\ldots, s\}\setminus u})} \right| & \le & q^{-\sum_{j=1}^s a_j} \sup_{\bsx \in \prod_{j=1}^s [r_j q^{-a_j},(r_j+1)q^{-a_j})} |\Delta(f,J_{\bsx})|.
\end{eqnarray*}
The result follows.
\end{proof}

In the following lemma we now obtain a bound on the Walsh
coefficients for functions of bounded variation. It is a
generalization of \cite[Proposition~6]{pirsic}.
\begin{lemma}\label{lem_proppirs}
Let $0 < \lambda \le 1$ and let $f \in \LL_2([0,1)^s)$ satisfy
$V_{\lambda,1,1,\bsgamma}(f) < \infty$. Then for any $\bsk \in
\NN_0^s\setminus\{\bszero\}$ the $\bsk$-th Walsh coefficient of $f$
satisfies $$|\hat{f}(\bsk)| \le q^{|u|-\lambda\sum_{j\in u} (a_j-1)}
V^{(|u|)}_{\lambda,1}(f_u;u),$$ where $\bsk = (k_1,\ldots, k_s)$, $u
= \{1\le j \le s: k_j \neq 0\}$ and for $j \in u$ we have $k_j =
\kappa_j q^{a_j-1} + k'_j$, where $\kappa_j \in \{1,\ldots, q-1\}$,
$a_j \in \NN$ and $0 \le k'_j < q^{a_j-1}$.
\end{lemma}

\begin{proof}
Let $f \in \LL_2([0,1)^s)$ with $\bsk$-th Walsh coefficient
$\hat{f}(\bsk)$. First note that it suffices to show the result for
$\bsk \in \NN^s$, as otherwise we only need to replace the function
$f$ with the function $f_u(\bsx_u) = \int_{[0,1)^{s-|u|}} f(\bsx)
\rd\bsx_{\cS\setminus u}$. Hence let now $\bsk\in\NN^s$ be given and
let $\bsk' = (k'_1,\ldots, k'_s)$. Then we have
\begin{eqnarray*}
\lefteqn{ \left|\int_{[0,1)^s} f(\bsx) \;\overline{\wal_{\bsk}(\bsx)} \rd \bsx\right| } \\ & \le & \sum_{0 \le c_j < q^{a_j-1} \atop 1 \le j \le s} \left|\int_{c_1 q^{-a_1+1}}^{(c_1+1) q^{-a_1+1}} \cdots \int_{c_s q^{-a_s+1}}^{(c_s+1) q^{-a_s+1}} f(\bsx) \; \overline{\wal_{\bsk}(\bsx)}\rd\bsx\right|.
\end{eqnarray*}

Now we use Lemma~\ref{lem_pirs} and thereby obtain that the above sum is bounded by $$\sum_{0 \le c_1 < q^{a_1-1}} \cdots \sum_{0 \le c_s < q^{a_s-1}} q^{-\sum_{j=1}^s (a_j-1)} \sup_{J} |\Delta(f,J)|,$$ where the supremum is taken over all boxes $J = \prod_{j=1}^s [d_j, e_j) \subseteq \prod_{j=1}^s [c_j q^{-a_j+1}, (c_j+1) q^{-a_j+1})$ with $q^{a_j}|e_j-d_j| \in \{1,\ldots, q-1\}$. Now we have $$q^{-\sum_{j=1}^s (a_j-1)} \sup_{J} |\Delta(f,J)| \le \sup_{\Pcal_{\bsc}} \sum_{I\in \Pcal_{\bsc}} \Vol(I)^{1-\lambda}|\Delta(f,I)| \frac{q^{-\sum_{j=1}^s (a_j-1)}}{\Vol(I)^{1-\lambda}},$$ where the supremum on the right hand side is taken over all partitions $\Pcal_{\bsc}$ of the cube $\prod_{j=1}^s [c_j q^{-a_j+1}, (c_j+1) q^{-a_j+1})$ and where each $I \in \Pcal_{\bsc}$ is of the form $I = \prod_{j=1}^s [x_j,y_j)$ with $q^{a_j}|y_j-x_j| \in \{1,\ldots, q-1\}$. We have $q^{-\sum_{j=1}^s a_j} \le \Vol(I) \le q^{-\sum_{j=1}^s (a_j-1)}$ and therefore $$\frac{q^{-\sum_{j=1}^s (a_j-1)}}{\Vol(I)^{1-\lambda}} \le \Vol(I)^\lambda q^s \le q^{s-\lambda\sum_{j=1}^s (a_j-1)}$$ and hence
\begin{equation*}
q^{-\sum_{j=1}^s (a_j-1)} \sup_{J} |\Delta(f,J)| \le q^{s-\lambda\sum_{j=1}^s (a_j-1)} \sup_{\Pcal_{\bsc}} \sum_{I\in \Pcal_{\bsc}} \Vol(I)^{1-\lambda}|\Delta(f,I)|.
\end{equation*}
Note that
\begin{equation*}
\sum_{0 \le c_j < q^{a_j-1} \atop 1 \le j \le s}  \sup_{\Pcal_{\bsc}} \sum_{I\in \Pcal_{\bsc}} \Vol(I)^{1-\lambda}|\Delta(f,I)| \le   \sup_{\Pcal} \sum_{J\in \Pcal} \Vol(J) \frac{|\Delta(f,J)|}{\Vol(J)^\lambda}
\end{equation*}
where the supremum on the left hand side is taken over all partitions $\Pcal$ of the cube $\prod_{j=1}^s [c_j q^{-a_j+1}, (c_j+1) q^{-a_j+1})$ into subintervals and the supremum on the right hand side is taken over all partitions of $[0,1)^s$ into subintervals. Thus the result follows.
\end{proof}

For the next lemma we will need the following two functions. For $\kappa\in\{1,\ldots, q-1\}$ let now $$\upsilon_{\kappa} = \sum_{r=0}^{q-1} r \wal_{\kappa}(r/q).$$ If $q$ is chosen to be a prime number and the bijections $\varphi$ and $\eta$ are chosen to be the identity, then $\upsilon_{\kappa} = q (\de^{2\pi\icomp \kappa/q}-1)^{-1}$, see \cite[Appendix~A]{DP05}.

Further for $l \in \{1,\ldots, q-1\}$ we define the function
$\zeta_a(x) = \sum_{r = 0}^{x_a-1} \overline{\wal_{l}(r/q)}$, where
$a \ge 1$ and $x = x_1 q^{-1} + x_2 q^{-2} + \cdots$ and where for
$x_a = 0$ we set $\zeta_a(x) = 0$. The function $\zeta_a$ depends on
$x$ only through $x_a$, thus it is a step-function which is constant
on the intervals $[c q^{-a}, (c+1)q^{-a})$ for $c = 0,\ldots,
q^a-1$. By \cite[Proposition~5]{pirsic} it follows that $\zeta_a$
can be represented by a finite Walsh series. Indeed, there are
numbers $c_0,\ldots, c_{q-1}$ (which depend on $l$ but not on $a$)
such that
$$\zeta_a(x) = \sum_{z=0}^{q-1} c_z \overline{\wal_{zq^{a-1}}(x)}.$$
If $q$ is chosen to be a prime number and the bijections $\varphi$
and $\eta$ are chosen to be the identity, then $\zeta_a(x) =
(1-\overline{\wal_{lq^{a-1}}(x)}) (1-\overline{\wal_l(1/q)})^{-1}$,
i.e., $c_0 = (1-\overline{\wal_l(1/q)})^{-1}$, $c_l =
(\overline{\wal_l(1/q)}-1)^{-1}$ and $c_z = 0$ for $z \neq 0,l$.

The following lemma will be used in the induction step for
differentiable functions. For example, for a differentiable function
$F:\real \rightarrow \real$ given by $F(x) = \int_0^x f(y) \rd y$ we
can calculate the Walsh coefficients using integration by parts in
the following way: for $k
> 0$ we have
\begin{eqnarray}\label{eq_indstep}
\hat{F}(k) & = & \int_0^1 F(x) \overline{\wal_k(x)} \rd x  = \left[
\int_0^x \overline{\wal_k(y)} \rd y F(x) \right]_0^1 - \int_0^1
 f(x) \int_0^x \overline{\wal_k(y)} \rd y \rd x \nonumber \\
& = & - \int_0^1 f(x) \int_0^x \overline{\wal_k(y)} \rd y \rd x,
\end{eqnarray}
where we used $\int_0^0 \overline{\wal_k(x)} \rd x = \int_0^1
\overline{\wal_k(x)} \rd x = 0$. For $k = 0$ on other hand we obtain
\begin{eqnarray}\label{eq_indstep0}
\hat{F}(0) & = & \int_0^1 F(x) \overline{\wal_0(x)} \rd x  = \left[
\int_0^x \overline{\wal_0(y)} \rd y F(x) \right]_0^1 - \int_0^1
 f(x) \int_0^x \overline{\wal_0(y)} \rd y \rd x \nonumber \\
& = &  \int_0^1 f(x) \rd x - \int_0^1 f(x) \int_0^x
\overline{\wal_0(y)} \rd y \rd x,
\end{eqnarray}

Thus if we know the Walsh series for $f$, then we can easily
calculate the Walsh series for $F$, provided that we know the Walsh
series for $\int_0^x \overline{\wal_k(y)} \rd y$. This will be
calculated in the following lemma. It appeared in a simpler form in
\cite{Fine}.
\begin{lemma}\label{lem_intJ}
For $k \in \NN_0$ and $x \in [0,1)$ define $J_{k}(x) = \int_{0}^x \overline{\wal_{k}(y)} \rd y$. For $k \ge 1$ let $k = l q^{a-1} + k'$ where $l \in \{1,\ldots, q-1\}$, $a \ge 1$ and $0 \le k' < q^{a-1}$. Then $J_{k}$ can be represented by a Walsh series which is given by
\begin{equation*}
J_{k}(x) = q^{-a}\Bigg( \sum_{z=0}^{q-1} c_z \overline{\wal_{z q^{a-1} + k'}(x)} + 2^{-1} \overline{\wal_{k}(x)} + \sum_{c=1}^\infty \sum_{\kappa=1}^{q-1} q^{-c-1} \upsilon_{\kappa} \overline{\wal_{\kappa q^{a+c-1} + k}(x)}\Bigg).
\end{equation*}
Further we have $$J_0(x) = 1/2 + \sum_{c=1}^\infty \sum_{\kappa=1}^{q-1} q^{-c-1} \upsilon_{\kappa} \; \overline{\wal_{\kappa q^{c-1}}(x)}.$$
\end{lemma}

\begin{proof}
Let $k = l q^{a-1} + k'$ with $a \ge 1$, $0 \le k' < q^{a-1}$ and $l \in \{1,\ldots, q-1\}$. The function $\overline{\wal_{l q^{a-1}}(y)}$ is constant on each interval $[r q^{-a}, (r+1) q^{-a})$ and $\overline{\wal_{k'}(y)}$ is constant on each interval $[c q^{-a+1}, (c+1) q^{-a+1})$. We have $\overline{\wal_k(y)} = \overline{\wal_{l q^{a-1}}(y)} \; \overline{\wal_{k'}(y)}$.
For any $0 \le c < q^{a-1}$ we have
\begin{eqnarray*}
\int_{[c q^{-a+1}, (c+1) q^{-a+1})} \!\!\!\!\overline{\wal_k(y)}\rd y & = & \overline{\wal_{k'}(cq^{-a+1})} \int_{[c q^{-a+1}, (c+1) q^{-a+1})} \!\!\!\!\overline{\wal_{l q^{a-1}}(y)}\rd y \\ & = & \overline{\wal_{k'}(c q^{-a+1})} q^{-a} \sum_{r = 0}^{q-1} \wal_l(r/q) \\ & = & 0.
\end{eqnarray*}

Thus we have $$J_k(x)  = \overline{\wal_{k'}(x)} J_{l q^{a-1}}(x).$$
Let $x = x_1 q^{-1} + x_2 q^{-2} + \cdots$ and $y = x_{a+1} q^{-1} +
x_{a+2} q^{-2} + \cdots$, then we have $$J_{l q^{a-1}}(x) = q^{-a}
\sum_{r = 0}^{x_a-1} \overline{\wal_{l}(r/q)} + q^{-a}
\overline{\wal_{l}(x_a/q)} y.$$

We now investigate the Walsh series representation of the function $J_{l q^{a-1}}(x)$. First note that  $\overline{\wal_l(x_a/q)} = \overline{\wal_{lq^{a-1}}(x)}$. Further, by a slight adaption of \cite[eq. (30)]{DP05} we obtain
\begin{equation}\label{eq_xwalsh}
y = 1/2 + \sum_{c=1}^\infty \sum_{\kappa=1}^{q-1} q^{-c-1} \upsilon_{\kappa} \; \overline{\wal_{\kappa q^{c-1}}(y)}.
\end{equation}
As $\overline{\wal_{\kappa q^{c-1}}(y)} = \overline{\wal_{\kappa q^{a+c-1}}(x)}$ we obtain $$y = 1/2 + \sum_{c=1}^\infty \sum_{\kappa=1}^{q-1} q^{-c-1} \upsilon_{\kappa} \; \overline{\wal_{\kappa q^{a+c-1}}(x)}.$$

As noted above, the Walsh series of $\zeta_a(x) = \sum_{r = 0}^{x_a-1} \overline{\wal_{l}(r/q)}$, where for $x_a = 0$ we set $\zeta_a(x) = 0$ can be written as $$\zeta_a(x) = \sum_{z=0}^{q-1} c_z \overline{\wal_{zq^{a-1}}(x)}.$$

Altogether we obtain
\begin{eqnarray*}
q^a J_{lq^{a-1}}(x) & = & \sum_{z=0}^{q-1} c_z \overline{\wal_{z q^{a-1}}(x)} + 2^{-1} \overline{\wal_{l q^{a-1}}(x)} \\ && \qquad\quad  + \sum_{c=1}^\infty \sum_{\kappa=1}^{q-1} q^{-c-1} \upsilon_{\kappa} \overline{\wal_{\kappa q^{a+c-1} + l q^{a-1}}(x)}
\end{eqnarray*}
and therefore
\begin{equation*}
q^a J_{k}(x) = \sum_{z=0}^{q-1} c_z \overline{\wal_{z q^{a-1} + k'}(x)} + 2^{-1} \overline{\wal_{k}(x)} + \sum_{c=1}^\infty \sum_{\kappa=1}^{q-1} q^{-c-1} \upsilon_{\kappa} \overline{\wal_{\kappa q^{a+c-1} + k}(x)}.
\end{equation*}
The result for $k = 0$ follows easily from (\ref{eq_xwalsh}).
\end{proof}

Note that Lemma~\ref{lem_intJ} can easily be generalized to
arbitrary dimensions $s$, since for $\bsk = (k_1,\ldots, k_s) \in
\NN_0^s$ we have for any $\bsx = (x_1,\ldots, x_s) \in [0,1)^s$ that
$$J_{\bsk}(\bsx) = \int_{[0,\bsx)} \overline{\wal_{\bsk}(\bsy)} \rd
\bsy = \prod_{j=1}^s J_{k_j}(x_j),$$ where $[0,\bsx) = \prod_{j=1}^s
[0,x_j)$.

The next lemma shows how the Walsh coefficients of a function $F =
\int f$ can be obtained from the Walsh coefficients of $f$.
\begin{lemma}\label{lem_sumfF}
Let $f\in \LL_2([0,1)^s)$ and let $F(\bsx) = \int_{[0,\bsx)}
f(\bsy)\rd\bsy$, where $[0,\bsx) = \prod_{j=1}^s [0,x_j)$ with $\bsx
= (x_1,\ldots, x_s)$. Further let $\hat{F}(\bsk)$ denote the
$\bsk$-th Walsh coefficient of $F$. Let $\bsk = (k_1,\ldots, k_s)
\in \NN_0^s$ and let $U = \{1\le j \le s: k_j \neq 0\}$. For $j\in
U$ let $k_j = l_j q^{a_j-1} + k_j'$, $0 < l_j < q$ and $0 \le k_j' <
q^{a_j-1}$ and further let $\bsk' = (k'_1,\ldots, k'_s)$ where $k'_j
=0$ for $j \notin U$. Then we have
\begin{eqnarray*}
\hat{F}(\bsk) & = & q^{-\sum_{j\in U} a_j} \sum_{U \subseteq v
\subseteq \cS} (-1)^{|v|} \sum_{\bsh_v \in \NN_0^{|v|}}
\hat{f}(\bsk' + (\bsh_v,\bszero)) \; \chi_{U,v,\bsk}(\bsh_v),
\end{eqnarray*}
where $(\bsh_v,\bszero)$ denotes the $s$-dimensional vector whose
$j$-th component is $h_j$ for $j \in v$ and $0$ otherwise and where
for $\bsh_v = (h_j)_{j\in v} \in \NN_0^{|v|}$ we set $$\chi_{U,v,
\bsk}(\bsh_v) = \prod_{j \in U} \rho_{k_j}(h_j) \prod_{j \in
v\setminus U} \phi(h_j).$$ Here
$$\rho_{k_j}(h_j) = \left\{\begin{array}{ll} c_z + 2^{-1} 1_{z = l_j}
& \mbox{for } h_j = z q^{a_j-1}, \\ \upsilon_z q^{-i-1} & \mbox{for
} h = z q^{a_j-1+i} + l_j q^{a_j-1}, i > 0, 0 < z < q,
\\ 0 & \mbox{otherwise,} \end{array} \right.$$
where $1_{z = l_j} = 1$ for $z = l_j$ and $0$ otherwise, and
$$\phi(h_j) = \left\{\begin{array}{ll} 2^{-1}
& \mbox{for } h_j = 0, \\ \upsilon_z q^{-i-1} & \mbox{for } h = z
q^{a_j-1+i}, i > 0, 0 < z < q,
\\ 0 & \mbox{otherwise.} \end{array} \right.$$
\end{lemma}

\begin{proof}
Using integration by parts in each coordinate, Fubini's theorem and
$J_k(0) = J_k(1) = 0$ for any $k \in\NN$ (see
Equations~(\ref{eq_indstep}) and (\ref{eq_indstep0}) for
one-dimensional examples) it follows that
$$\hat{F}(\bsk) = \int_{[0,1)^s} F(\bsx)
\overline{\wal_{\bsk}(\bsx)} \rd \bsx = \sum_{U \subseteq v
\subseteq \cS} (-1)^{|v|} \int_{[0,1)^s} J_{\bsk_v}(\bsx_v) f(\bsx)
\rd \bsx,$$ where for $\bsk = (k_1,\ldots, k_s)$ and $\bsx =
(x_1,\ldots, x_s)$ we have $J_{\bsk_v}(\bsx_v) = \prod_{j \in v}
J_{k_j}(x_j)$.

Using the Walsh series expansion of $J_{\bsk}$ given by
Lemma~\ref{lem_intJ} we can now express the Walsh coefficient
$\hat{F}(\bsk)$ as a sum of the Walsh coefficients $\hat{f}(\bsh)$,
from which the result follows.
\end{proof}

The following definition now captures the essence of the decay of
the Walsh coefficients of smooth functions and will be used in the
statement of the subsequent lemmas, theorems and corollaries.

\begin{definition}\label{def_qadicnonincreasing}
Let $k = k(v;a_1,\ldots, a_v) = \kappa_1 q^{a_1-1} + \cdots +
\kappa_v q^{a_v-1}$ with $v \ge 1$, $\kappa_1,\ldots, \kappa_v \in
\{1,\ldots, q-1\}$ and $1 \le a_v < \cdots < a_1$ be a natural
number. For $k = 0$ we set $v = 0$, i.e., $k(0) = 0$. A function
$\BB:\NN_0 \rightarrow \real$ is called $q$-adically non-increasing
if $\BB(k) = \BB(k(v;a_1,\ldots,a_v))$ is non-increasing in $v$ and
each $a_i$ for $i = 1,\ldots, v$, that is, for any $v \ge 0$ we have
$$\BB(k(v;a_1,\ldots, a_v)) \ge \BB(k(v+1;a'_1,\ldots, a'_{v+1}))$$
with $1\le a'_{v+1} < \cdots < a'_{1}$ and $a_1,\ldots, a_v \in
\{a'_1,\ldots, a'_{v+1}\}$ and for an arbitrary $1 \le i \le v$ we
have $$\BB(k(v;a_1,\ldots, a_v)) \ge \BB(k(v;a_1,\ldots, a_{i-1},
a_i+1, a_{i+1}, \ldots, a_v))$$ provided that $a_i+1 < a_{i-1}$ in
case $1 < i \le v$.
\end{definition}

In the following lemma we give a bound on the Walsh coefficients of $F$ if $f$ satisfies some smoothness condition.
\begin{lemma}\label{lem_iteration}
Let $\BB:\NN_0^s \rightarrow [0,\infty)$ be a $q$-adically
non-increasing function in each variable. Let $f\in \LL_2([0,1)^s)$
and let the Walsh coefficients of $f$ satisfy
$$|\hat{f}(\bsk)| \le \BB(\bsk) \quad \mbox{for all } \bsk \in
\NN_0^s.$$

Let $F(\bsx) = \int_{[0,\bsx)} f(\bsy)\rd\bsy$, where $[0,\bsx) =
\prod_{j=1}^s [0,x_j)$ with $\bsx = (x_1,\ldots, x_s)$. Further let
$\hat{F}(\bsk)$ denote the $\bsk$-th Walsh coefficient of $F$. Let
$\bsk = (k_1,\ldots, k_s) \in \NN_0^s \setminus\{\bszero\}$ and let
$U = \{1\le j \le s: k_j \neq 0\}$. For $j\in U$ let $k_j = l_j
q^{a_j-1} + k_j'$ and $0 \le k_j' < q^{a_j-1}$ and further let
$\bsk' = (k'_1,\ldots, k'_s)$ where $k'_j =0$ for $j \notin U$. Then
there is a constant $C_{s,U}
> 0$ independent of $\bsk$ such that
$$|\hat{F}(\bsk)| \le C_{s,U} \; q^{-\sum_{j\in U} a_j}
\BB(\bsk').$$
\end{lemma}

\begin{proof}
Using Lemma~\ref{lem_sumfF} we obtain that $$|\hat{F}(\bsk)| \le
q^{-\sum_{j\in U} a_j} \BB(\bsk') \sum_{U\subseteq v \subseteq \cS}
\sum_{\bsh_v \in \NN_0^{|v|}} |\chi_{U,v,\bsk}(\bsh_v)|,$$ as
$|\hat{f}(\bsk'+(\bsh_v,\bszero))| \le \BB(\bsk'+(\bsh_v,\bszero))
\le \BB(\bsk')$ for all values of $\bsh_v \in \NN_0^{|v|}$ for which
$\chi_{U,v,\bsk}(\bsh_v) \neq 0$, since $\BB$ is $q$-adically
non-increasing in each variable.

Thus it remains to bound $\sum_{U\subseteq v \subseteq \cS}
\sum_{\bsh_v \in \NN_0^{|v|}} |\chi_{U,v,\bsk}(\bsh_v)|$
independently of $\bsk$. We only prove the case where $q$ is chosen
to be a prime number and the bijections $\varphi$ and $\eta$ are
chosen to be the identity, as in this case we can obtain an explicit
constant $C_{s,U} > 0$. The general case can be obtained by similar
arguments using the result from Lemma~\ref{lem_sumfF}.

Using the notation from Lemma~\ref{lem_sumfF} we have
\begin{eqnarray*}
\sum_{h\in\NN_0} |\rho_{k_j}(h)| & = & |1-\omega_q^{-l_j}|^{-1} + 2^{-1} |1+\omega_q^{-l_j}| |\omega_q^{-l_j}-1|^{-1} \\ &&  + \sum_{i=1}^\infty q^{-i} \sum_{z=1}^{q-1} |\de^{2\pi\icomp z/q}-1|^{-1} \\
& \le & 3 (2-2\cos(2\pi/q))^{-1/2}
\end{eqnarray*}
and
$$\sum_{h=0}^\infty |\phi(h)|  = 2^{-1} + \sum_{i=1}^\infty q^{-i} \sum_{z=1}^{q-1} |\de^{2\pi\icomp z/q}-1|^{-1} \le 2^{-1} + (2-2\cos(2\pi/q))^{-1/2}.$$
Therefore we have
\begin{eqnarray*}
\lefteqn{ \sum_{U\subseteq v \subseteq S} \sum_{\bsh_v \in \NN_0^{|v|}} |\chi_{U,v,\bsk}(\bsh_v)| } \\ & \le & 3^{|U|} (2-2\cos(2\pi/q))^{-|U|/2} (3/2+(2-2\cos(2\pi/q))^{-1/2})^{s-|U|}
\end{eqnarray*}
and hence we can choose
\begin{equation}\label{eq_constcsu}
C_{s,U} = 3^{|U|} (2-2\cos(2\pi/q))^{-|U|/2} (3/2+(2-2\cos(2\pi/q))^{-1/2})^{s-|U|}
\end{equation}
in Lemma~\ref{lem_iteration} for this case.
\end{proof}

We use the above results now to establish an upper bound on the Walsh coefficients of a polynomial. The proof will give a glimpse on how the argument will work for more general function classes.
\begin{lemma}\label{lem_walshpoly}
Let $k = \kappa_1 q^{a_1-1} + \cdots + \kappa_v q^{a_v-1}$ with $v\ge 1$, $\kappa_1,\ldots, \kappa_v \in \{1,\ldots, q-1\}$ and $1 \le a_v < \cdots < a_1$. For $v = 0$ let $k = 0$. Let $f:[0,1)\rightarrow \real$ be the polynomial $f(x) = f_0 + f_1 x + \cdots + f_i x^i$ with $f_i \neq 0$ and let $\hat{f}(k)$ denote the $k$-th Walsh coefficient of $f$. Then for $v \ge 0$ there are constants $0 < C_{f,i,v} < \infty$ such that $$|\hat{f}(k)| \le C_{f,i,v} q^{-a_1 - \cdots - a_v},$$ where we can choose $C_{f,i,v} = 0$ for $v > i$.
\end{lemma}
\begin{proof}
Let $f(x) = f_0 + f_1 x + \cdots + f_i x^i$, where $i = \deg(f)$
(that is, $f_i \neq 0$). Then we have $f^{(i)}(x) = i! f_i \neq 0$.
As $f^{(i)}$ is a constant function, its Walsh series representation
is simply given by $f^{(i)}(x) = i! f_i$. Now we use
Lemma~\ref{lem_iteration}. The dimension $s$ in our case is $1$ and
we can choose the function $\BB_1$ by $\BB_1(0) = i!|f_i|$ and for
$k > 0$ we set $\BB_1(k) = 0$. Note the function $\BB_1$ defined
this way is a $q$-adically non-increasing function. Then it follows
that there is a constant $C_1 > 0$ such that the Walsh coefficients
of the function $\int_0^x f^{(i)}(t) \rd t = f^{(i-1)}(x) -
f^{(i-1)}(0)$ are bounded by $C_1 q^{-a_1} i!|f_i|$ for all $k$
where $v = 1$ and the Walsh coefficients are $0$ for $v > 1$. The
Walsh coefficient for $k = 0$ is given by $f^{(i-2)}(1) -
f^{(i-2)}(0) - f^{(i-1)}(0)$.

Now consider the function $\int_0^x f^{(i-1)}(t)\rd t = f^{(i-2)}(x) - f^{(i-2)}(0)$. It follows from the above and Lemma~\ref{lem_iteration} that the Walsh coefficients of $ f^{(i-2)}(x) - f^{(i-2)}(0)$ can be bounded by a $q$-adically non-increasing function $\BB_2$. Indeed there are constants $C_2,C_3 > 0$ such that we can choose $\BB_2(0) = |f^{(i-2)}(1) - f^{(i-2)}(0) - f^{(i-1)}(0)|$, $\BB_2(k) =  C_{2} q^{-a_1}$ for $v = 1$, $\BB_2(k) = C_3 q^{-a_1-a_2} $ for $v = 2$ and $\BB_2(k) = 0$ for $v > 2$. Again $\BB_2$ is a $q$-adically non-increasing function and Lemma~\ref{lem_iteration} can again be used.

By using the above argument iteratively we obtain that there is a constant $C > 0$ such that $|\hat{f}(k)| \le C q^{-a_1 - \cdots - a_v}$, for $k =\kappa_1 q^{a_1-1} + \cdots + \kappa_v q^{a_v-1}$ with $\kappa_1,\ldots, \kappa_v \in \{1,\ldots, q-1\}$ and $1 \le a_v < \cdots < a_1$. The result thus follows.
\end{proof}

For the case where $q$ is chosen to be a prime number and the bijections $\varphi$ and $\eta$ are chosen to be the identity and for $0 \le v \le i$ we can choose
\begin{equation}\label{eq_constcfv}
C_{f,i,v} = \bar{C}^v \sum_{l = v}^i C'^{i-l}  l! |f_l|,
\end{equation}
where $\bar{C} = (2-2\cos(2\pi/q))^{-1/2}$ and $C' = 3/2 + (2-2\cos(2\pi/q))^{-1/2}$ in Lemma~\ref{lem_walshpoly}.

Let $f:[0,1)^s \rightarrow \real$ be such that the partial mixed
derivatives up to order $\delta\ge 1$ in each variable exist and are
continuous. We need some further notation: let $\bstau =
(\tau_1,\ldots, \tau_s)$ and $$f^{(\bstau)}(\bsx) =
\frac{\partial^{\tau_1+\cdots + \tau_s}}{\partial x_1^{\tau_1}
\cdots \partial x_s^{\tau_s}} f(\bsx).$$ For $\bstau \in \{0,\ldots,
\delta\}^s$ let $u(\bstau) = \{1\le j \le s: \tau_j = \delta\}$. Let
$\bsgamma = (\gamma_v)_{v\subset \NN}$ be an indexed set of
non-negative real numbers. Let $v(\bstau) = \{1\le j \le s: \tau_j >
0\}$. Then for $0 < \lambda \le 1$ and $\pp,\qq,\rr \ge 1$
($\pp,\qq,\rr$ do not appear in the subscript of $N$ as they do not
have influence on our subsequent bounds, we only assume that they
are bigger or equal to $1$) we define
\begin{equation}\label{def_N}
N_{\delta, \lambda,\bsgamma}(f) = \left(\sum_{\bstau \in \{0,\ldots,
\delta\}^s} \gamma_{v(\bstau)}^{-1}
\left[V_{\lambda,\pp,\qq,\bsone}^{(|u(\bstau)|)}(f^{(\bstau)}(\cdot,
\bszero_{\cS\setminus u(\bstau)}))\right]^\rr \right)^{1/\rr},
\end{equation}
where, for clarity, we introduce the additional superscript
$(|u(\bstau)|)$ in the Hardy and Krause variation
$V_{\lambda,\pp,\qq,\bsone}^{(|u(\bstau)|)}$ which indicates the
dimension of the function and where for $u(\bstau) = \emptyset$ we
set $V_{\lambda,\pp,\qq,\bsone}^{(|u(\bstau)|}(f^{(\bstau)}(\cdot,
\bszero_{\cS\setminus u(\bstau)})) = |f^{(\bstau)}(\bszero)|$.

The weights $\bsgamma$ are introduced to modify the importance of
various coordinate projections and were first introduced in
\cite{SW98}, see also \cite{DSWW1,DSWW2}. If for some $v' \subseteq
\cS$ the weight $\gamma_{v'} = 0$, then we assume that the function
$f$ satisfies
$V_{\lambda,\pp,\qq,\bsone}^{(|v'|)}(f^{(\bstau)}(\cdot,
\bszero_{\cS\setminus v'})) = 0$ for all $\bstau \in \{0,\ldots,
\delta\}^s$ with $v(\bstau) = v'$ and in (\ref{def_N}) we formally
set $0/0 = 0$.

The parameters in the definition of $N_{\delta, \lambda,\bsgamma}$
have the following meaning:
\begin{itemize}
\item $\delta$ denotes the order of partial derivatives of $f$ required in order for $N_{\delta,
\lambda,\bsgamma}(f)$ to make sense;
\item $\lambda$ is a H\"older type parameter or fractional order type
parameter of the generalized Hardy and Krause variation; roughly,
$f$ needs to have partial derivatives up to order $\delta +
\lambda$, where for $0 < \lambda < 1$ this means some type of
fractional smoothness or in dimension one a H\"older condition of
order $\lambda$;
\item the Vitali variation is in $\pp$ norm;
\item $\qq$ is the norm in the summation of the generalized Hardy and Krause
variation;
\item $\rr$ is the norm in the summation over the $\bstau$;
\item $\bsgamma$ are the weights which regulate the importance of different coordinate projections;
\end{itemize}

Note that for $\lambda = 1$ and $\pp = \qq = \rr = 2$ the functional
$N_{\delta, \lambda,\bsgamma}$ is just the norm in a weighted
reproducing kernel Sobolev space with continuous partial mixed
derivatives up to order $\delta + 1$ in each variable. In one
dimension the unweighted norm in this reproducing kernel Sobolev
space is given by
\begin{eqnarray}\label{eq_ipsob1}
\lefteqn{\langle f, g \rangle_{\soban,\delta+1} } \\ &=& f(0) g(0) + \cdots + f^{(\delta-1)}(0) g^{(\delta-1)}(0) \nonumber \\ &&  + \int_0^1 f^{(\delta)}(x) \rd x \int_0^1 g^{(\delta)}(x)\rd x + \int_0^1 f^{(\delta+1)}(x) g^{(\delta+1)}(x) \rd x \nonumber
\end{eqnarray}
and for higher dimensions one just takes the weighted tensor product
of the one dimensional reproducing kernel Sobolev spaces (see
\cite{DSWW2} for examples of weighted tensor product reproducing
kernel Sobolev spaces). Let the $s$ dimensional weighted inner
product be denoted by $\langle \cdot, \cdot
\rangle_{\soban,s,\delta+1, \bsgamma}$ and the corresponding norm
$\|\cdot\|_{\soban,s,\delta+1,\bsgamma}$ (indeed if the partial
mixed derivatives up to order $\delta+1$ of $f$ are continuous on
$[0,1]^s$ then we have $\|f\|_{\soban,s,\delta+1,\bsgamma} =
N_{\delta, 1,\bsgamma}(f)$).

In the following we define a function $\mu$ which will be used
throughout the paper: let $\delta \ge 1$ be an integer and $0 <
\lambda \le 1$ be a real number. Then for $\bsk = (k_1,\ldots, k_s)$
we set
\begin{equation}\label{defmus} \mu_{q,\delta+\lambda}(\bsk) = \sum_{j=1}^s
\mu_{q,\delta+\lambda}(k_j)\end{equation} with
\begin{equation}\label{defmu} \mu_{q,\delta + \lambda}(k) =
\left\{\begin{array}{ll} 0 & \mbox{for } k = 0,
\\ a_1 + \cdots + a_v & \mbox{for } v \le \delta, \\  a_1 + \cdots + a_{\delta} +  \lambda a_{\delta +1} & \mbox{for } v > \delta,
\end{array} \right.
\end{equation}
where for $k \in \NN$ we write  $k = \kappa_1 q^{a_1 -1} + \cdots +
\kappa_v q^{a_v-1}$ with $v \ge 1$, $\kappa_1,\ldots, \kappa_v \in
\{1,\ldots, q-1\}$ and $1 \le a_v < \cdots < a_1$.

\begin{theorem}\label{th_boundwalshcoeff}
Let $\delta \ge 1$ be an integer, $0 < \lambda \le 1$, $\pp, \qq,
\rr \ge 1$ be real numbers and an indexed set $\bsgamma =
(\gamma_v)_{v\subset\NN}$ of non-negative real numbers be given. Let
$f:[0,1)^s \rightarrow \real$ be such that the partial mixed
derivatives up to order $\delta$ in each variable exist and such
that $N_{\delta,\lambda,\bsgamma}(f) < \infty$. Then for any $\bsk
\in \NN_0^s\setminus\{\bszero\}$ it follows that there is a constant
$C_{f,q,s,\bsgamma} >0$ independent of $\bsk$ such that
$$|\hat{f}(\bsk)| \le C_{f,q,s,\bsgamma}
q^{-\mu_{q,\delta+\lambda}(\bsk)}.$$
\end{theorem}

\begin{proof}
In order to prove the result we use the Taylor series expansion of the function $f$. We have
\begin{eqnarray}\label{eq_taylor}
f(\bsx) &=& \sum_{\bstau \in \{0,\ldots, \delta-1\}^s}
\frac{\bsx^{\bstau}}{\bstau!} f^{(\bstau)}(\bszero)  +
\sum_{\emptyset\neq u \subseteq\cS}\sum_{\bstau_{\cS\setminus u} \in
\{0,\ldots,\delta-1\}^{s-|u|}}  ((\delta-1)!)^{-|u|} \nonumber \\ &&
\frac{\prod_{j\in\cS\setminus u}
x_j^{\tau_j}}{\prod_{j\in\cS\setminus u} \tau_j!}
\int_{[\bszero_u,\bsx_u)} f^{(\bsdelta_u, \bstau_{\cS\setminus
u})}(\bsy_u, \bszero_{\cS\setminus u}) \prod_{j\in u}
(x_j-y_j)^{\delta-1} \rd \bsy_u.
\end{eqnarray}

First note that the first sum in (\ref{eq_taylor}) is a polynomial
in $\bsx$ and therefore the Walsh coefficients of this polynomial
satisfy the desired bound by Lemma~\ref{lem_walshpoly}.

Now we consider the second sum. Let $\emptyset \neq u \subseteq \cS$
with $u = \{j_1,\ldots, j_{|u|}\}$ be given. Then for $j \notin u$
the Walsh coefficients satisfy the desired bound by
Lemma~\ref{lem_walshpoly}. Hence it remains to consider the Walsh
coefficients of
\begin{eqnarray*}
G_u(\bsx_u)  &=& \int_{[\bszero_u,\bsx_u)} f^{(\bsdelta_u,
\bstau_{\cS\setminus u})}(\bsy_u, \bszero_{\cS\setminus u})
\prod_{j\in u} (x_j-y_j)^{\delta-1} \rd \bsy_u \\ & = &
\int_{[\bszero_u,\bsone_u)} f^{(\bsdelta_u, \bstau_{\cS\setminus
u})}(\bsy_u, \bszero_{\cS\setminus u}) \prod_{j\in u}
(x_j-y_j)_+^{\delta-1} \rd \bsy_u,
\end{eqnarray*}
where $(x-y)_+ = \max(0, x-y)$.

By differentiating the function $G_u$ in each variable $0 \le k < \delta$ times we obtain (see \cite[pp. 153,154]{SB})
\begin{eqnarray*}
\lefteqn{\frac{\partial^{k|u|}}{\partial \bsx_u^{k}} G_u(\bsx_u) }
\\ & =& \left(\frac{(\delta-1)!}{(\delta-1-k)!}\right)^{|u|}
\int_{[\bszero_u,\bsx_u)}  f^{(\bsdelta_u, \bstau_{S\setminus
u})}(\bsy_u, \bszero_{S\setminus u}) \prod_{j\in u}
(x_j-y_j)^{\delta-1-k} \rd \bsy_u.
\end{eqnarray*}
Hence $\frac{\partial^{k|u|}}{\partial \bsx_u^{k}} G_u(\bsx_u) = 0$ if there is at least one $j \in u$ such that $x_j = 0$. Further we have $$\frac{\partial^{\delta|u|}}{\partial \bsx_u^{\delta}} G_u(\bsx_u)  =  ((\delta-1)!)^{|u|}  f^{(\bsdelta_u, \bstau_{S\setminus u})}(\bsx_u, \bszero_{S\setminus u}).$$

From $N_{\delta,\lambda,\bsgamma}(f) < \infty$ it follows that the
Walsh coefficients of $f^{(\bsdelta_u,\bstau_{\cS\setminus
u})}(\bsx_u,\bszero_{\cS\setminus u})$ decay with order
$\mu_{q,0+\lambda}(k)$ in each variable. Further we have
$$G_u(\bsx_u) = \int_{[\bszero,\bsx_u)} \int_{[\bszero,\bsy_1)}
\cdots \int_{[\bszero,\bsy_{\delta-1})}
G_u^{(\bsdelta_u)}(\bsy_{\delta}) \rd\bsy_{\delta} \cdots \rd
\bsy_1$$ as the function $G_u$ and its derivatives are $0$ if at
least one $x_j = 0$ for $j \in u$, i.e., we have
$$\int_{[\bszero,\bsx_u)} G_u^{(\bstau)}(\bsy)\rd\bsy =
\sum_{v\subseteq u} (-1)^{u\setminus v}
G_u^{(\bstau-\bsone)}(\bsx_v,\bszero_{u\setminus v}) =
G_u^{(\bstau-\bsone)}(\bsx_u).$$ Hence it follows by repeated use of
Lemma~\ref{lem_iteration} that the desired bound holds for $G_u$ and
thus the result follows from (\ref{eq_taylor}).
\end{proof}

For the case where $q$ is chosen to be a prime number and the bijections $\varphi$ and $\eta$ are chosen to be the identity we can also obtain an explicit constant in Theorem~\ref{th_boundwalshcoeff}. Indeed, using Lemma~\ref{lem_proppirs}, Lemma~\ref{lem_iteration} together with the explicit constant (\ref{eq_constcsu}) and Lemma~\ref{lem_walshpoly} together with the explicit constant (\ref{eq_constcfv}) we obtain that the constant $C_{q,\bsgamma}$ can be chosen as
\begin{eqnarray*}
C_{f,q,s,\bsgamma} &=&  \sum_{\bstau \in \{0,\ldots, \delta-1\}^s\atop \gamma_{v(\bstau)} \neq 0} |f^{(\bstau)}(\bszero)| \hat{C}^{\tau_1 + \cdots + \tau_s} + \sum_{\emptyset \neq u \subseteq \cS} q^{|u|} C_{s,u}^\delta \\ && \sum_{\bstau_{\cS\setminus u} \in \{0,\ldots,\delta-1\}^{s-|u|} \atop \gamma_{u\cup v(\bstau_{\cS\setminus u})} \neq 0} \!\!\! \hat{C}^{\sum_{j \in \cS\setminus u} \tau_j}  V^{(|u|)}_{\lambda,1}(f^{(\bsdelta_u,\bstau_{\cS\setminus u})}(\cdot, \bszero_{\cS\setminus u})), \nonumber
\end{eqnarray*}
where $\hat{C} = 1$ for $2 \le q < 6$ and $\hat{C} = (2-2\cos(2\pi/q))^{-1/2}$ for $q > 6$ (note that for $q > 6$ we have $\hat{C} > 1$) and $$C_{s,u} = 3^{|u|} (2-2\cos(2\pi/q))^{-|u|/2} (3/2+(2-2\cos(2\pi/q))^{-1/2})^{s-|u|}.$$ As noted above, under certain conditions we can write $V^{(|u|)}_{\lambda,1}$ also as an integral, see (\ref{eq_formelV}).

We give a further useful estimation of the constant by separating the dependence of the function from the constants. This way we obtain
\begin{equation*}
C_{f,q,s,\bsgamma} \le C_{\delta,q,s,\bsgamma}
N_{\delta,\lambda,\bsgamma}(f),
\end{equation*}
where
\begin{eqnarray}\label{eq_constcn}
C_{\delta,q,s,\bsgamma} &=&  \sum_{\bstau \in \{0,\ldots, \delta-1\}^s} \gamma_{v(\bstau)} \hat{C}^{\tau_1 + \cdots + \tau_s} \nonumber \\ && + \sum_{\emptyset \neq u \subseteq \cS} q^{|u|} C_{s,u}^\delta \sum_{\bstau_{\cS\setminus u} \in \{0,\ldots,\delta-1\}^{s-|u|}} \gamma_{u\cup v(\bstau_{\cS\setminus u})} \hat{C}^{\sum_{j \in \cS\setminus u} \tau_j}.
\end{eqnarray}
Consider now the case where the weights are of product form (see \cite{DSWW2}), i.e., there is a sequence of positive real numbers $(\gamma_j)_{j\in\NN}$ such that $\gamma_v = \prod_{j\in v} \gamma_j$ for all $v \subset \NN$ and for $v = \emptyset$ we set $\gamma_v = 1$. If now $q$ is prime with $2 \le q < 6$, then
\begin{eqnarray*}
C_{\delta,q,s,\bsgamma} & = & \prod_{j=1}^s (1 + \gamma_j (\delta-1)) + \prod_{j=1}^s \left[(1+\gamma_j(\delta-1)) (3/2+\hat{C}) + \gamma_j 3 q \hat{C}\right] \\ && - \prod_{j=1}^s \left[(1+\gamma_j(\delta-1))(3/2+\hat{C})\right].
\end{eqnarray*}
For example for $q = 2$, $\delta = 1$ and product weights we obtain $$C_{1,2,s,\bsgamma} = 1 - 2^s + 2^s \prod_{j=1}^s (1 + 6 \gamma_j).$$

The approach used here for prime $q$ and $\varphi$ the identity map can also be used for arbitrary prime powers $q$ and arbitrary mappings $\varphi$ with $\varphi(0) = 0$. Hence we obtain the following corollary.

\begin{corollary}\label{cor_boundwalshcoeff}
Under the assumptions of Theorem~\ref{th_boundwalshcoeff} there
exists a constant $C_{\delta,q,s,\bsgamma} > 0$ independent of
$\bsk$ and $f$ such that  $$|\hat{f}(\bsk)| \le C_{\delta,
q,s,\bsgamma} N_{\delta,\lambda,\bsgamma}(f)
q^{-\mu_{q,\delta+\lambda}(\bsk)} \quad \mbox{ for all } \bsk \in
\NN_0^s,$$ where $\mu_{q,\delta+\lambda}$ is given by (\ref{defmus})
and (\ref{defmu}).
\end{corollary}

\begin{remark}\rm
The results in this section also hold for the following
generalization. In the definition of $N_{\delta,\lambda,\bsgamma}$
we anchored the function and its derivatives at $0$, i.e., we used
$V_{\lambda,\pp, \qq,
\bsone}^{(|u(\bstau)|)}(f(\cdot,\bszero_{\cS\setminus u(\bstau)}))$.
This can be generalized by choosing an arbitrary $\bsa \in [0,1]^s$
and using
$V_{\lambda,\pp,\qq,\bsone}^{(|u(\bstau)|)}(f(\cdot,\bsa_{S\setminus
u(\bstau)}))$ in the definition of $N_{\delta,\lambda,\bsgamma}$. It
can be shown that in this case we also have
Theorem~\ref{th_boundwalshcoeff} and
Corollary~\ref{cor_boundwalshcoeff}.
\end{remark}

\subsection{Convergence of the Walsh
series}\label{subsect_convergence}

For our purposes here we need strong assumptions on the convergence
of the Walsh series $S(f)(\bsx) = \sum_{\bsk \in \NN_0^s}
\hat{f}(\bsk) \wal_{\bsk}(\bsx)$ to the function $f$, i.e., we
require that the partial series $S_{m}(f)(\bsx) = \sum_{\bsk \in
\NN_0^s \atop k_j < m} \hat{f}(\bsk) \wal_{\bsk}(\bsx)$ converges to
$f(\bsx)$ at every point $\bsx \in [0,1)^s$ as $m \rightarrow
\infty$. (Note that the Walsh series $S(f)$ for the functions
considered in this paper is always absolutely convergent, i.e.,
$\sum_{\bsk\in\NN_0^s} |\hat{f}(\bsk)| < \infty$, hence the Walsh
series $S(f)(\bsx)$ is uniformly bounded by $\sum_{\bsk\in\NN_0^s}
|\hat{f}(\bsk)|$ and therefore $S(f)(\bsx)$ itself converges at
every point $\bsx \in [0,1)^s$.) This is necessary as we want to
approximate the integral at function values $\bsx_n$ and for our
analysis we deal with the Walsh series rather than the function
itself, hence it is paramount that the function and its Walsh series
coincide at every point $\bsx \in [0,1)^s$.

As the functions considered here are at least differentiable it
follows that they are continuous and using the argument in \cite[p.
373]{Fine} it follows that the Walsh series really converges at
every point $\bsx \in [0,1)^s$ to the function value $f(\bsx)$.
Indeed, for a given $\bsx \in [0,1)^s$ we have $$S_{q^m}(f)(\bsx) =
\sum_{\bsk \in \{0,\ldots, q^m-1\}^s} \hat{f}(\bsk)\wal_{\bsk}(\bsx)
= \Vol(J_{\bsx})^{-1} \int_{J_{\bsx}} f(\bsx) \rd \bsx,$$ where
$J_{\bsx} = \prod_{j=1}^s [q^{-m}\lfloor q^m x_j\rfloor,
q^{-m}\lfloor q^m x_j\rfloor+q^{-m})$. The  last equality follows
from
\begin{eqnarray*}
\lefteqn{\sum_{\bsk \in \{0,\ldots, q^m-1\}^s}
\hat{f}(\bsk)\wal_{\bsk}(\bsx) } \qquad\qquad \\ & = &
\int_{[0,1)^s} f(\bsy) \sum_{\bsk\in \{0,\ldots, q^m-1\}^s}
\wal_{\bsk}(\bsx) \overline{\wal_{\bsk}(\bsy)} \rd \bsy \\ & = &
\Vol(J_{\bsx})^{-1} \int_{J_{\bsx}} f(\bsx) \rd \bsx.
\end{eqnarray*}
As the function $f$ is continuous it immediately follows that
$S_{q^m}(f)(\bsx)$ converges to $f(\bsx)$ as $m$ goes to infinity
and the result follows.

\subsection{A function space based on Walsh functions containing smooth functions}\label{subsectwalshspaceE}

In this section we use the above results to define a function space
based on Walsh functions which contains smooth functions for
smoothness conditions considered in the previous section.

Let $\vartheta > 1$ be a real number and $q$ a prime power.  Then
for $\bsk \in \NN_0^s$ we set $r_{q,\vartheta}(\bsk) =
q^{-\mu_{q,\vartheta}(\bsk)}$, where $\mu_{q,\vartheta}$ is given by
(\ref{defmus}) and (\ref{defmu}) (if $\vartheta$ is an integer, then
choose $\lambda = 1$ and $\delta = \vartheta-1$ and otherwise
$\delta = \lfloor \vartheta \rfloor$ and $\lambda = \vartheta -
\lfloor \vartheta \rfloor$).

Now we define a function space $\EE_{s,q,\vartheta,\bsgamma}
\subseteq \LL_2([0,1)^s)$ with norm
$\|\cdot\|_{\EE_{s,q,\vartheta,\bsgamma}}$ given by
$$\|f\|_{\EE_{s,q,\vartheta,\bsgamma}} = \max_{u \subseteq \cS \atop
\gamma_u \neq 0} \gamma_u^{-1} \sup_{\bsk_u \in \NN^{|u|}}
\frac{|\hat{f}(\bsk_u,\bszero_{\cS\setminus
u})|}{r_{q,\vartheta}(\bsk_u)},$$ where again for $\gamma_u = 0$ we
assume that $\hat{f}(\bsk_u,\bszero_{\cS\setminus u}) = 0$ for all
$\bsk_u \in \NN^{|u|}$.

The following result follows now directly from Corollary~\ref{cor_boundwalshcoeff}.
\begin{corollary}\label{cor_smoothfwalsh}
Let $\delta \ge 1$, $0 < \lambda \le 1$, $\pp, \qq, \rr \ge 1$ and
an indexed set $\bsgamma = (\gamma_v)_{v\subset\NN}$ of non-negative
real numbers be given. Then there exists a constant
$C_{\delta,q,s,\bsgamma} > 0$ such that for every function
$f:[0,1)^s \rightarrow \real$, whose partial mixed derivatives up to
order $\delta$ exist, we have
$$\|f\|_{\EE_{s,q,\delta+\lambda,\bsgamma}} \le C_{\delta,
q,s,\bsgamma} N_{\delta,\lambda,\bsgamma}(f),$$ where
$\mu_{q,\delta+\lambda}$ is given by (\ref{defmus}) and
(\ref{defmu}).
\end{corollary}

Again, using (\ref{eq_constcn}) an explicit constant in
Corollary~\ref{cor_smoothfwalsh} can be obtained for $q$ prime and
$\varphi$ the identity map. For all other cases (i.e., arbitrary
prime powers $q$ and mappings $\varphi$ with $\varphi(0) = 0$)
explicit constants can be obtained as well, but in this case the
constant may also depend on the particular choice of $q$ and
$\varphi$. Further, as noted already above, for $\lambda = 1$ and
$\pp = \qq = \rr = 2$ the functional $N_{\delta,\lambda,\bsgamma}$
coincides with the norm in a certain Sobolev space (the one
dimensional inner product for this Sobolev space is given by
(\ref{eq_ipsob1}) and for higher dimensions one just considers
tensor products of the one dimensional space) and hence it follows
that $\EE_{s,q,\delta + 1,\bsgamma}$ contains certain Sobolev
spaces. Hence Corollary~\ref{cor_smoothfwalsh} shows that if we want
to prove results for smooth functions it is enough to consider
$\EE_{s,q,\vartheta,\bsgamma}$ (in the following we design
quasi-Monte Carlo rules which work well for
$\EE_{s,q,\vartheta,\bsgamma}$ rather than directly for smooth
functions, so the results for smooth functions come as a byproduct).

A function $f \in \EE_{s,q,\vartheta,\bsgamma}$ can be written as a
sum of their anova terms $f = \sum_{u \subseteq \cS} f_u$ (see
\cite{ES}). For a function $f\in \EE_{s,q,\vartheta,\bsgamma}$ given
by $f(\bsx) = \sum_{\bsk \in \NN_0^s} \hat{f}(\bsk)
\wal_{\bsk}(\bsx)$ the anova term $f_u$ corresponding to a subset $u
\subseteq \cS$ is simply given by $$f_u(\bsx_u) = \sum_{\bsk_u \in
\NN^{|u|}} \hat{f}(\bsk_u,\bszero_{\cS\setminus u})
\wal_{\bsk_u}(\bsx_u).$$ If for some $u \subseteq \cS$ we have
$\gamma_u = 0$, then this implies that the anova term corresponding
to $u$ satisfies $f_u \equiv 0$. Hence the Walsh space
$\EE_{s,q,\vartheta, \bsgamma}$ consists only of functions whose
anova term belonging to a subset $u$ is zero for all subsets $u$
with $\gamma_u = 0$ (see also \cite{DSWW2}).

\section{Digital $(t,\alpha, \beta ,n\times m,s)$-nets and digital $(t,\alpha,\beta, \sigma ,s)$-sequences}\label{sectdignets}

In this section we give the definition of digital $(t,\alpha, \beta,n\times m,s)$-nets and digital $(t,\alpha, \beta,\sigma, s)$-sequences. Similar point sets were introduced in \cite{Dick05}.

\subsection{The digital construction scheme}

The construction of the point set used here is a slight generalization of the digital construction scheme introduced by Niederreiter, see \cite{niesiam}, by breaking with the tradition of having square generating matrices.

\begin{definition}\label{def_digcons} \rm
Let $q$ be a prime-power and let $n,m,s \ge 1$ be integers. Let
$C_1,\ldots ,C_s$ be $n \times m$ matrices over the finite field
$\FF_q$ of order $q$. Now we construct $q^m$ points in $[0,1)^s$:
for $0 \le h \le q^m -1$ let $h=h_0+h_1 q +\cdots +h_{m-1} q^{m-1}$
be the $q$-adic expansion of $h$. Consider an arbitrary but fixed
bijection $\varphi:\{0,1,\ldots ,q-1\}\longrightarrow \FF_q$.
Identify $h$ with the vector $\vec{h}=(\varphi(h_0),\ldots
,\varphi(h_{m-1}))^{\top} \in \FF_q^m$, where $\top$ means the
transpose of the vector (note that we write $\vec{h}$ for vectors in
the finite field $\FF_q^m$ and $\bsh$ for vectors of integers or
real numbers). For $1 \le j \le s$ multiply the matrix $C_j$ by
$\vec{h}$, i.e.,
$$C_j \vec{h}=:(y_{j,1}(h),\ldots ,y_{j,n}(h))^{\top} \in \FF_q^n,$$
and set
$$x_{h,j}:=\frac{\varphi^{-1}(y_{j,1}(h))}{q}+\cdots
+\frac{\varphi^{-1}(y_{j,n}(h))}{q^n}.$$  The point set $\{\bsx_0,\ldots, \bsx_{q^m-1}\}$ is called a digital net (over $\FF_q$) (with generating matrices $C_1,\ldots, C_s$).

For $n,m = \infty$ we obtain a sequence $\{\bsx_0,\bsx_1,\ldots\}$, which is called a digital sequence (over $\FF_q$) (with generating matrices $C_1,\ldots, C_s$).
\end{definition}

Niederreiter's concept of a digital $(t,m,s)$-net and a digital $(t,s)$-sequence will appear as a special case in the subsequent section. Further, the digital nets considered below all satisfy $n \ge m$.

For a digital net with generating matrices $C_1,\ldots, C_s$ let $\Dcal = \Dcal(C_1,\ldots, C_s)$ be the dual net given by $$\Dcal = \{\bsk \in \NN_0^s\setminus \{\bszero\}: C_1^\top \vec{k}_1 + \cdots + C_s^\top \vec{k}_s = \vec{0}\},$$ where for $\bsk = (k_1,\ldots, k_s)$ with $k_j = \kappa_{j,0} + \kappa_{j,1} q + \cdots$ and $\kappa_{j,i} \in \{0,\ldots, q-1\}$ let $\vec{k}_j = (\varphi(\kappa_{j,0}),\ldots, \varphi(\kappa_{j,n-1}))^\top$. Further, for $\emptyset \neq u \subseteq \cS$ let $\Dcal_u = \Dcal((C_j)_{j\in u})$ and $\Dcal_u^\ast = \Dcal_u \cap \NN^{|u|}$.

Note that throughout the paper Walsh functions and digital nets are defined using the same finite field $\FF_q$ and the same bijection $\varphi$.

The following lemma is a slight generalization of \cite[Lemma~2.5]{PDP}.

\begin{lemma} \label{fqwalsum}
 Let $\{\bsx_0,\ldots,\bsx_{q^m-1} \}$ be a digital net over
$\FF_{q}$ with bijection $\varphi$, where $\varphi(0)=0$, generated by the $n\times m$ matrices $C_1,\ldots, C_s$ over $\FF_q$, $n,m\ge 1$. Then for any vector $\bsk=(k_1,\ldots ,k_s)$ of
nonnegative integers $0\leq k_1,\ldots,k_s<q^n$ we have
\[ \sum_{h=0}^{q^m-1}\, _{\FF_q,\varphi}\wal_{\bsk}(\bsx_h) =
\begin{cases}
q^m &
\text{ if } \bsk \in \Dcal \cup \{\bszero\}, \\
0 & \text{ else}, \end{cases}
\]
where $\bszero$ is the zero vector in $\NN_0^s$.
\end{lemma}

\subsection{$(t,\alpha,\beta,n \times m,s)$-nets and $(t,\alpha,\beta,\sigma,s)$-sequences}\label{sec_talpha}

Digital $(t,\alpha,\beta,m,s)$-nets and digital $(t,\alpha,\beta,s)$-sequences were first introduced in \cite{Dick05}. Those point sets were used for quasi-Monte Carlo rules which achieve the optimal rate of convergence of the worst-case error in Korobov spaces (which are reproducing kernel Hilbert spaces of smooth periodic functions). By a slight generalization of digital $(t,\alpha,\beta,m,s)$-nets we will show that those digital nets also achieve the optimal convergence of the worst-case error in the space $\EE_{s,q,\vartheta,\bsgamma}$ for all $1 < \vartheta \le \alpha$.

The $t$ value of a $(t,m,s)$-net is a quality parameter for the
distribution properties of the net. A low $t$ value yields well
distributed point sets and it has been shown, see for example
\cite{DP05,niesiam}, that a small $t$ value also guarantees a small
worst-case error for integration in Sobolev spaces for which the
partial first derivatives are square integrable. In \cite{Dick05} it
was shown how the definition of the $t$ value needs to be modified
in order to obtain faster convergence rates for periodic Sobolev
spaces for which the partial derivatives up to order $\delta \le
\beta$ are square integrable. Here we extend those result in several
ways. First we generalize the digital $(t,\alpha,\beta,m,s)$-nets
used in \cite{Dick05} to digital $(t,\alpha,\beta,n\times m,s)$-nets
and show that we then can remove the periodicity assumption
necessary in \cite{Dick05}. Further, if the derivatives up to order
$\delta$ also have bounded variation with coefficient $0 < \lambda
\le 1$, then we have shown that such functions are in
$\EE_{s,q,\delta+\lambda,\bsgamma}$.

In the following we repeat some definitions and results from
\cite{Dick05} and give the definition of digital
$(t,\alpha,\beta,n\times m,s)$-nets and digital $(t,\alpha,\beta,
\sigma, s)$-sequences.

For a real number $\vartheta > 1$ the definition of the Walsh space
$\EE_{s,q,\vartheta, \bsgamma}$ suggests to define the metric
$\mu_{q,\vartheta}(\bsk,\bsl) = \mu_{q,\vartheta}(\bsk\ominus \bsl)$
on $\NN_0^s$, where $\mu_{q,\vartheta}(\bsk \ominus\bsl)$ is given
by (\ref{defmus}) and (\ref{defmu}), which is an extension of the
metric introduced in \cite{nie86}, see also \cite{rt} (the metric
for $\vartheta = 1$ can be used for Walsh spaces for example
considered in \cite{DP05}; for this case one basically obtains the
metric in \cite{nie86,rt}).  As we will see later, in order to
obtain a small worst-case error in the Walsh space
$\EE_{s,q,\vartheta,\bsgamma}$ we need digital nets for which
$\min\{\mu_{q,\vartheta}(\bsk): \bsk \in \Dcal\}$ is large. By
translating this property into a linear independence property of the
row vectors of the generating matrices $C_1,\ldots, C_s$ we arrive
at the following definition.

\begin{definition}\rm\label{def_net}
Let $n, m,\alpha \ge 1$ be natural numbers, let $0 <\beta \le \alpha
m/n$ be a real number and let $0 \le t \le \beta n$  be  a natural
number. Let $\FF_q$ be the finite field of prime power order $q$ and
let $C_1,\ldots, C_s \in \FF_q^{n \times m}$ with $C_j = (c_{j,1},
\ldots, c_{j,n})^\top$. If for all $1 \le i_{j,\nu_j} < \cdots <
i_{j,1} \le n$, where $0 \le \nu_j \le m$ for all $j = 1,\ldots, s$,
with $$\sum_{j = 1}^s \sum_{l=1}^{\min(\nu_j,\alpha)} i_{j,l}  \le
\beta n - t$$ the vectors
$$c_{1,i_{1,\nu_1}}, \ldots, c_{1,i_{1,1}}, \ldots,
c_{s,i_{s,\nu_s}}, \ldots, c_{s,i_{s,1}}$$ are linearly independent
over $\FF_q$ then the digital net with generating matrices
$C_1,\ldots, C_s$ is called a digital $(t,\alpha,\beta,n\times
m,s)$-net over $\FF_q$. Further we call a digital $(t,\alpha,\beta,
n\times m,s)$-net over $\FF_q$ with the largest possible value of
$\beta$, i.e., $\beta = \alpha m/n$, a digital $(t,\alpha,n\times
m,s)$-net over $\FF_q$.

If $t$ is the smallest non-negative integer such that the digital
net generated by $C_1,\ldots, C_s$ is a digital
$(t,\alpha,\beta,n\times m,s)$-net, then we call the digital net a
strict digital $(t,\alpha,\beta, n \times m,s)$-net or a strict
digital $(t,\alpha,n \times m,s)$-net if $\beta = \alpha m/n$.
\end{definition}

\begin{remark}\rm
Using duality theory (see \cite{np}) it follows that for a digital $(t,\alpha,\beta,n \times m,s)$-net we have $\min_{\bsk \in \Dcal} \mu_{q,\alpha}(\bsk) > \beta n - t$ and for a strict digital $(t,\alpha,\beta,n \times m,s)$-net we have $\min_{\bsk \in \Dcal} \mu_{q,\alpha}(\bsk) = \beta n - t + 1$. Hence digital $(t,\alpha,\beta,n \times m,s)$-nets with high quality have a large value of $\beta n - t$.
\end{remark}

\begin{remark}\rm
In summary the parameters $t, \alpha,\beta, n, m, s$ have the
following meaning:
\begin{itemize}
\item $s$ denotes the dimension of the point set.
\item $n$ and $m$ denote the size of the generating matrices for
digital nets, i.e. the generating matrices are of size $n \times m$;
in particular this means the point set has $q^m$ points.
\item $t$ denotes the quality parameter of the point set; a low $t$
value means high quality. In the upper bound, $t$ is a quality
parameter related to the constant in the upper bound.
\item $\beta$ is also a quality parameter. We will see later that the integration error is roughly $q^{-n}$.
This is of course only true within boundaries, which is the reason
for the parameter $\beta$, i.e. the integration error is roughly
$q^{-\beta n}$. Hence $\beta$ is a quality parameter related to the
convergence rate.
\item $\alpha$ is the smoothness parameter of the point set.
\end{itemize}

We can group the parameters also in the following way:
\begin{itemize}
\item $m,n,s$ are fixed parameters, i.e. they specify the number and
size of the generating matrices.
\item $\alpha$ is a variable parameter, i.e. given (fixed) generating
matrices can for example generate a $(t_1, 1, \beta_1, 10 \times 5,
5)$-net, a $(t_2, 2, \beta_2, 10 \times 5, 5)$-net, and so on (note
the point set is always the same in each instance; the values
$t_1,t_2, \ldots, \beta_1,\beta_2, \ldots$ may differ). This is
necessary as in the upper bounds $\alpha$ will be the smoothness of
the integrand, which may not be known explicitly.
\item $t$ and $\beta$ are dependent parameters, they will depend on
the generating matrices and on $\alpha$. For given generating
matrices, it is desirable to know the values of $\beta$ and $t$ for
each value of $\alpha \in \NN$.
\end{itemize}
\end{remark}

Digital $(t,\alpha,\beta,n\times m,s)$-nets do not exist for
arbitrary choices of the parameters $t,\alpha,\beta,n,m,s$, see
\cite{Dick05}. The digital nets considered in \cite{Dick05} had the
restriction that $n = m$ and special attention was paid to those
digital nets with high quality, i.e., where $\alpha = \beta$. In
this paper, a special role will be played by those digital nets for
which $n = \alpha m$ and $\beta = 1$. The restriction on the linear
independence of the digital nets comprises now $n - t = \alpha m -
t$ row vectors, which is the same as in \cite{Dick05}, with the only
difference that the size of the generating matrices is now bigger as
now each generating matrix has $n = \alpha m$ rows. As those digital
nets play a special role in this work we have the following
definition.
\begin{definition}\rm\label{def_netn}
A digital $(t,\alpha,1, \alpha m \times m, s)$-net over $\FF_q$ is
called a digital $(t,\alpha, \alpha m \times m,s)$-net over $\FF_q$.
A strict digital $(t,\alpha,1, \alpha m \times m, s)$-net over
$\FF_q$ is called a strict digital $(t,\alpha, \alpha m \times
m,s)$-net over $\FF_q$.
\end{definition}

\begin{remark}\rm
For practical purposes we would like to explicitly know digital
$(t,\alpha,\alpha m \times m, s)$-nets for all $\alpha,m,s \ge 1$
with $t$ as small as possible (as will be shown later, they achieve
the optimal rate of convergence of the integration error of
integrands for which all mixed partial derivatives of order $\alpha$
are, for example, square integrable, thus their usefulness).

Further, for given $\alpha,m,s \ge 1$ and a given digital
$(t,\alpha, \alpha m \times m,s)$-net $P$, we would then also like
to know the $t'$ and $\beta'$ value of this point set $P$ when
viewed as a digital $(t',\delta,\beta', \alpha m \times m, s)$-net
for all values $\delta \in \NN$, i.e., $t'$ and $\beta'$ are
functions of $\delta$ (this is because we would also like to know
how well such a digital net $P$ performs if the integrand has
partial mixed derivatives of order up to $\delta$, because we might
not know the smoothness of the integrand, but still would wish that
$P$ performs best possible).
\end{remark}

We can also define sequences of points for which the first $q^m$
points form a digital $(t,\alpha,\beta,n\times m,s)$-nets. In the
classical case \cite{niesiam} one can just consider the left-upper
$m \times m$ submatrices of the generating matrices of a digital
sequence and determine the net properties of these for each $m \in
\NN$. Here, on the other hand, we are considering digital nets whose
generating matrices are $n \times m$ matrices. So we would have to
consider the left-upper $n_m \times m$ submatrices of the generating
matrices of the digital sequence for each $m \in \NN$ and where
$(n_m)_{m \in\NN}$ is a sequence of natural numbers. For our
purposes here it is enough to consider only $n_m$ of the form
$\sigma m$, for some given $\sigma \in \NN$.

\begin{definition}\rm\label{def_seq}
Let $\alpha, \sigma \ge 1$ and $t \ge 0$ be integers and let $0 <\beta \le \alpha/\sigma$ be a real number. Let $\FF_q$ be the finite field of prime power order $q$ and let $C_1,\ldots, C_s \in \FF_q^{\infty \times \infty}$ with $C_j = (c_{j,1}, c_{j,2}, \ldots)^\top$. Further let $C_{j,\sigma m \times m}$ denote the left upper $\sigma m \times m$ submatrix of  $C_j$. If for all $m > t/(\beta\sigma)$ the matrices $C_{1,\sigma m \times m},\ldots, C_{s,\sigma m \times m}$ generate a digital $(t,\alpha,\beta,\sigma m \times m,s)$-net then the digital sequence with generating matrices $C_1,\ldots, C_s$ is called a digital $(t,\alpha,\beta, \sigma ,s)$-sequence over $\FF_q$. Further we call a digital $(t,\alpha,1, \alpha,s)$-sequence over $\FF_q$ a digital $(t,\alpha,s)$-sequence over $\FF_q$.

If $t$ is the smallest non-negative integer such that the digital sequence generated by $C_1,\ldots, C_s$ is a digital $(t,\alpha,\beta, \sigma,s)$-sequence, then we call the digital sequence a strict digital $(t,\alpha,\beta, \sigma, s)$-sequence or a strict digital $(t,\alpha,s)$-sequence if $\alpha = \sigma$ and $\beta = 1$.
\end{definition}

For short we will often write $(t,\alpha,\beta, n\times m, s)$-net
instead of digital $(t,\alpha,\beta, n\times m, s)$-net over
$\FF_q$. The same applies to the other notions defined above.

\begin{remark}\rm
Note that the definition of a digital $(t, 1,m\times m,s)$-net
coincides with the definition of a digital $(t,m,s)$-net and the
definition of a digital $(t, 1,s)$-sequence coincides with the
definition of a digital $(t,s)$-sequence as defined by
Niederreiter~\cite{niesiam}. Further note that the $t$-value depends
on $\alpha, \beta$ and $\sigma$, i.e., $t = t(\alpha,\beta,\sigma)$
or $t = t(\alpha)$ if $\alpha = \sigma$ and $\beta = 1$.
\end{remark}

The definition of $(t,\alpha,s)$-sequences here differs slightly
from the definition in \cite{Dick05}. Indeed the definition of a
$(t,\alpha,s)$-sequence in \cite{Dick05} corresponds to a
$(t,\alpha,\alpha,1,s)$-sequence in the terminology of this paper,
whereas here we call a $(t,\alpha,1,\alpha,s)$-sequence a
$(t,\alpha,s)$-sequence. On the other hand note that the condition
of linear independence in Definition~\ref{def_net} is the same in
both cases, i.e., the sum $i_{1,1} + \cdots +
i_{1,\min(\nu_1,\alpha)} + \cdots + i_{s,1} + \cdots +
i_{s,\min(\nu_s,\alpha)}$ needs to be bounded by $\alpha m -t$ for
all $m$ for $(t,\alpha,1,\alpha,s)$-sequences and also for
$(t,\alpha,\alpha,1,s)$-sequences.

\subsection{Some properties of $(t,\alpha,\beta, n\times m, s)$-nets and $(t,\alpha,\beta,\sigma,s)$-sequences}

The properties of such digital nets and sequences shown in
\cite{Dick05} also hold here. For example it was shown there that a
digital $(t,\alpha,m,s)$-net is also a digital $(\lceil t
\alpha'/\alpha\rceil ,\alpha',m,s)$-net for all $1\le \alpha' \le
\alpha$ and every digital $(t,\alpha,s)$-sequence is also a digital
$(\lceil t\alpha'/\alpha \rceil,\alpha',s)$-sequence for all $1\le
\alpha' \le \alpha$. In the same way we have the following theorem.

\begin{theorem}\label{th_prop}
Let $P$ be a digital $(t,\alpha,\beta,n\times m,s)$-net over $\FF_q$
and let  $S$ be a digital $(t,\alpha,\beta, \sigma,s)$-sequence over
$\FF_q$. Then we have:
\begin{enumerate}
\item[(i)] $P$ is a digital $(t',\alpha,\beta',n\times m,s)$-net for all $1\le \beta' \le \beta$ and  all $t \le t' \le \beta' m$ and $S$ is a digital $(t',\alpha,\beta', \sigma ,s)$-sequence for all $1 \le \beta' \le \beta$ and all $t \le t'$.
\item[(ii)] $P$ is a digital $(t',\alpha',\beta',n\times m,s)$-net for all $1\le \alpha' \le n$ where $\beta' = \beta\min(\alpha,\alpha')/\alpha$ and $t' =  \lceil t \min(\alpha,\alpha')/\alpha\rceil$ and $S$ is a digital $(t',\alpha',\beta', \sigma,s)$-sequence for all $\alpha' \ge 1$ where  $\beta' = \beta \min(\alpha,\alpha')/\alpha$ and where $t' =  \lceil t \min(\alpha,\alpha')/\alpha\rceil$.
\item[(iii)] Any digital $(t,\alpha,n \times m,s)$-net is a digital $(\lceil t \alpha'/\alpha\rceil ,\alpha',n\times m,s)$-net for all $1\le \alpha' \le \alpha$ and every digital $(t,\alpha,\sigma, s)$-sequence is a digital $(\lceil t\alpha'/\alpha \rceil,\alpha', \sigma,s)$-sequence for all $1\le \alpha' \le \alpha$.
\item[(iv)] If $C_1,\ldots, C_s \in \integer_b^{n\times m}$ are the generating matrices of a digital $(t,\alpha,\beta,n\times m,s)$-net then the matrices $C_1^{(n')}, \ldots, C_s^{(n')}$, where $C_j^{(n')}$ consists of the first $n'$ rows of $C_j$, generate a digital $(t,\alpha,\beta, n'\times m,s)$-net for all $1 \le n' \le n$ .
\item[(v)] Any digital $(t,\alpha,\beta,\sigma,s)$-sequence is a digital $(t,\alpha,\beta,\sigma',s)$-sequence for all $1 \le \sigma' \le \sigma$.
\end{enumerate}
\end{theorem}

\subsection{Constructions of $(t,\alpha,\beta, n \times m,s)$-nets and  $(t,\alpha, \sigma,s)$-sequences}\label{sec_talphacons}

In this section we show how explicit examples of $(t,\alpha,\beta, n \times m,s)$-nets and  $(t,\alpha,\beta,\sigma, s)$-sequences can be constructed. The idea for the construction is based on the construction method presented in \cite{Dick05}.

Let $d \ge 1$ and let $C_1,\ldots, C_{sd}$ be the generating matrices of a digital $(t,m,sd)$-net. Note that many explicit examples of such generating matrices are known, see for example \cite{faure,niesiam,NX,sob67} and the references therein. For the construction of a $(t,\alpha,m,s)$-net any of the above mentioned explicit constructions can be used, but as will be shown below the quality of the $(t,\alpha,m,s)$-net obtained depends on the quality of the underlying digital $(t,m,sd)$-net on which our construction is based on.

 Let $C_j = (c_{j,1},\ldots, c_{j,m})^\top$ for $j = 1,\ldots, sd$, i.e., $c_{j,l}$ are the row vectors of $C_j$. Now let the matrix $C^{(d)}_{j}$ be made of the first rows of the matrices $C_{(j-1)d + 1},\ldots, C_{jd}$, then the second rows of $C_{(j-1)d+1},\ldots, C_{jd}$ and so on. The matrix $C^{(d)}_{j}$ is then an $d m \times m$ matrix, i.e., $C^{(d)}_j = (c^{(d)}_{j,1},\ldots, c^{(d)}_{j,dm})^\top$ where $c^{(d)}_{j,l} = c_{u,v}$ with $l = (v-j)d + u$, $1\le v \le m$ and $(j-1)d < u \le jd$ for $l = 1,\ldots, d m$ and $j = 1,\ldots, s$. The following result is a slight generalization of \cite[Theorem~3]{Dick05} and can be obtained using the same proof technique.

\begin{theorem}\label{th_talphabeta}
Let $d \ge 1$ be a natural number and let $C_{1},\ldots, C_{sd}$ be the generating matrices of a digital $(t',m,sd)$-net over the finite field $\FF_q$ of prime power order $q$. Let $C^{(d)}_{1},\ldots, C^{(d)}_{s}$ be defined as above. Then for any $\alpha \ge 1$ the matrices $C^{(d)}_{1},\ldots, C^{(d)}_{s}$ are  generating matrices of a digital $(t,\alpha,\min(1,\alpha/d),dm \times m,s)$-net over $\FF_q$ with $$t = \min(\alpha,d)\;t' + \left\lceil \frac{s(d-1) \min(\alpha,d)}{2}\right\rceil.$$
\end{theorem}

The above construction and Theorem~\ref{th_talphabeta} can easily be
extended to $(t,\alpha,\beta,\sigma,s)$-sequences. Indeed, let $d
\ge 1$ and let $C_1,\ldots, C_{sd}$ be the generating matrices of a
digital $(t,sd)$-sequence. Again many explicit generating matrices
are known, see for example \cite{faure,niesiam,NX,sob67}. Let $C_j =
(c_{j,1},c_{j,2},\ldots)^\top$ for $j = 1,\ldots,sd$, i.e.,
$c_{j,l}$ are the row vectors of $C_j$. Now let the matrix
$C^{(d)}_{j}$ be made of the first rows of the matrices $C_{(j-1)d +
1},\ldots, C_{jd}$, then the second rows of $C_{(j-1)d+1},\ldots,
C_{jd}$ and so on, i.e., $$C^{(d)}_{j} = (c_{(j-1)d+1,1},\ldots,
c_{jd,1},c_{(j-1)d+1,2},\ldots,c_{jd,2},\ldots)^\top.$$ The
following theorem states that the matrices $C^{(d)}_{1},\ldots,
C^{(d)}_{s}$ are the generating matrices of a digital
$(t,\alpha,\min(1,\alpha/d), d,s)$-sequence, compare with
\cite[Theorem~4]{Dick05}.

\begin{theorem}\label{th_talphabetaseq}
Let $d \ge 1$ be a natural number and let $C_{1},\ldots, C_{sd}$ be the generating matrices of a digital $(t',sd)$-sequence over the finite field $\FF_q$ of prime power order $q$. Let $C^{(d)}_{1},\ldots, C^{(d)}_{s}$ be defined as above. Then for any $\alpha \ge 1$ the matrices $C^{(d)}_{1},\ldots, C^{(d)}_{s}$ are generating matrices of a digital $(t,\alpha,\min(1,\alpha/d), d,s)$-sequence over $\FF_q$ with $$t = \min(\alpha,d)\;t' + \left\lceil \frac{s(d-1) \min(\alpha,d)}{2}\right\rceil.$$
\end{theorem}

The last result shows that $(t,\alpha,\beta, \sigma m \times
m,s)$-nets indeed exist for $\beta = 1$ and any $0 < \sigma \le
\alpha$ and for $m$ arbitrarily large. We have even shown that
digital $(t,\alpha,\beta, \alpha m \times m,s)$-nets exist which are
extensible in $m$ and $s$. This can be achieved by using an
underlying $(t',sd)$-sequence which is itself extensible in $m$ and
$s$. If the $t'$ value of the original $(t',m,s)$-net or
$(t',s)$-sequence is known explicitly then we also know the $t$
value of the digital $(t,\alpha,\beta,\alpha m \times m,s)$-net or
$(t,\alpha,\beta, \sigma ,s)$-sequence. Furthermore it has also been
shown how such digital nets can be constructed in practice. Further
results on such sequences are established in \cite{Dick05}.

\section{Numerical integration in the Walsh space $\EE_{s,q,\vartheta,\bsgamma}$}\label{sect_int}

In this section we investigate numerical integration in the Walsh
space $\EE_{s,q,\vartheta,\bsgamma}$ using quasi-Monte Carlo rules
$$Q_{q^m,s}(f) = \frac{1}{q^m} \sum_{n=0}^{q^m-1} f(\bsx_n),$$ where
$\bsx_0,\ldots, \bsx_{q^m-1}$ are the points of a digital
$(t,\alpha,\beta,m,s)$-net over $\FF_q$. More precisely, we want to
approximate the integral $$I_s(f) = \int_{[0,1]^s} f(\bsx)\rd\bsx$$
by the quasi-Monte Carlo rule $Q_{q^m,s}(f)$. As a quality measure
for our rule we introduce the worst-case error in the next section.

\subsection{The worst-case error in the Walsh space $\EE_{s,q,\vartheta,\bsgamma}$}

The worst-case error for the Walsh
space~$\EE_{s,q,\vartheta,\bsgamma}$ using the quasi-Monte Carlo
rule $Q_{q^m,s}$ is given by $$e(Q_{q^m,s},\EE_{s,
q,\vartheta,\bsgamma}) = \sup_{f \in \EE_{s,q,\vartheta,\bsgamma}
\atop \|f\|_{\EE_{s,q,\vartheta,\bsgamma}} \le 1} \left|I_s(f) -
Q_{q^m,s}(f)\right|.$$ The initial error is given by
$$e(Q_{0,s},\EE_{s,q,\vartheta,\bsgamma}) =
\sup_{f\in\EE_{s,q,\vartheta,\bsgamma} \atop
\|f\|_{\EE_{s,q,\vartheta,\bsgamma}}\le 1} \left|I_s(f)\right|.$$

In the following we use digital nets generated by the matrices
$C_1,\ldots, C_s$ as quadrature points for the quadrature rule
$Q_{q^m,s}$. Let $f \in \EE_{s,q,\vartheta,\bsgamma}$. Using
Lemma~\ref{fqwalsum} it follows that
\begin{eqnarray*}
|I_s(f) - Q_{q^m,s}(f)| & = & \left|\sum_{\bsk \in \Dcal} \hat{f}(\bsk) \right| \\ &\le & \sum_{\bsk \in \Dcal} |\hat{f}(\bsk)| = \sum_{\emptyset \neq u \subseteq \cS} \sum_{\bsk_u \in \Dcal_u^\ast} |\hat{f}(\bsk_u,\bszero_{\cS\setminus u})|.
\end{eqnarray*}
Now we have $|\hat{f}(\bsk_u,\bszero_{\cS\setminus u})| \le \gamma_u r_{q,\vartheta}(\bsk_u) \|f\|_{\EE_{s,q,\vartheta,\bsgamma}}$ and thus we obtain
\begin{equation}\label{eq_bounderror}
|I_s(f) - Q_{q^m,s}(f)| \le \|f\|_{\EE_{s,q,\vartheta,\bsgamma}} \sum_{\emptyset \neq u \subseteq \cS} \gamma_u \sum_{\bsk_u \in \Dcal_u^\ast} r_{q,\vartheta}(\bsk_u).
\end{equation}
By choosing $\hat{f}(\bsk_u,\bszero_{\cS\setminus u}) = \gamma_u
r_{q,\vartheta}(\bsk_u)$ for all $u$ and $\bsk_u$ we can also obtain
equality in (\ref{eq_bounderror}). Thus we have
\begin{equation}\label{wce_eq}
e(Q_{q^m,s},\EE_{s,q,\vartheta,\bsgamma}) =
\sum_{\emptyset \neq u \subseteq\cS} \gamma_u \sum_{\bsk_u \in \Dcal_u^\ast} r_{q,\vartheta}(\bsk_u).
\end{equation}
From the last formula we can now see that essentially a large value
of $\min\{\mu_{q,\vartheta}(\bsk): \bsk \in \Dcal\}$ guarantees a
small worst-case error. Further it can be shown that
\begin{equation}\label{initialerror} e(Q_{0,s},\EE_{s,q,\vartheta,\bsgamma}) =
\gamma_{\emptyset}.\end{equation} We have shown the following
theorem.

\begin{theorem}
The initial error for multivariate integration in the Walsh space
$\EE_{s,q,\vartheta,\bsgamma}$ is given by (\ref{initialerror}) and
the worst-case error for multivariate integration in the Walsh
space~$\EE_{s,q,\vartheta,\bsgamma}$ using a digital net as
quadrature points is given by (\ref{wce_eq}).
\end{theorem}

In the following lemma we establish an upper bound on the sum
$\sum_{\bsk_u \in \Dcal_u^\ast} r_{q,\vartheta}(\bsk_u)$ for digital
$(t,\alpha,\beta, n\times m,s)$-nets over $\FF_q$. The proof is
similar to \cite[Lemma~6]{Dick05}.
\begin{lemma}\label{lem_boundsum}
Let $\vartheta > 1$ be a real number, $q \ge 2$ be a prime power,
$C_1,\ldots, C_s \in \FF_q^{n\times m}$ be the generating matrices
of a digital $(t,\lceil \vartheta \rceil,\beta,n \times m,s)$-net
over $\FF_q$ with $0 < \beta \le 1$ and let $\Dcal_u^\ast =
\Dcal_u^\ast((C_j)_{j\in u})$. For all $\emptyset \neq u \subseteq
\cS$ we have: if $\vartheta$ is not an integer it follows that
$$\sum_{\bsk_u\in \Dcal_u^\ast} r_{q,\vartheta}(\bsk_u) \le
C_{|u|,q,\vartheta} (\beta n - t + \lceil\vartheta\rceil)^{|u|
\lceil\vartheta\rceil - 1} q^{-\vartheta\lfloor (\beta n -
t)/\lceil\vartheta \rceil \rfloor},$$ where $$C_{|u|,q,\vartheta} =
q^{|u|\lceil\vartheta\rceil}
((q-q^{\vartheta-\lfloor\vartheta\rfloor })^{-1} +
(1-q^{(1-\vartheta)/\lceil\vartheta\rceil})^{-|u|\lceil\vartheta\rceil})$$
and if $\vartheta$ is an integer it follows that $$\sum_{\bsk_u\in
\Dcal_u^\ast} r_{q,\vartheta}(\bsk_u) \le C'_{|u|,q,\vartheta}
(\beta n - t + \vartheta)^{|u| \vartheta} q^{-(\beta n - t)},$$
where $$C'_{|u|,q,\vartheta} = q^{|u|\vartheta} (q^{-1} +
(1-q^{1/\vartheta-1})^{-|u|\vartheta}).$$
\end{lemma}
\begin{proof}
To simplify the notation we prove the result only for $u = \cS$. For all other subsets the result follows by the same arguments.

We first consider the case where $\vartheta > 1$ is not an integer. We partition the set $\Dcal^\ast_{\cS}$ into parts where the highest digits of $k_j$ are prescribed and we count the number of solutions of $C_1^\top \vec{k}_1 + \cdots + C_s^\top \vec{k}_s = \vec{0}$.
For $j = 1,\ldots, s$ let now $i_{j,\lceil \vartheta \rceil} < \cdots < i_{j,1}$ with $i_{j,1} \ge 1$.  Note that we now allow $i_{j,l} < 1$, in which case the contributions of those $i_{j,l}$ are to be ignored. This notation is adopted in order to avoid considering many special cases. Further we write $\bsi_{s,\lceil \vartheta \rceil} = (i_{1,1},\ldots, i_{1,\lceil \vartheta\rceil},\ldots, i_{s,1},\ldots, i_{s,\lceil \vartheta \rceil})$ and define
\begin{eqnarray*}
\Dcal^\ast_{\cS}(\bsi_{s,\lceil\vartheta\rceil}) &=& \{\bsk \in \Dcal^\ast_{\cS}: k_j = \lfloor \kappa_{j,1} q^{i_{j,1}-1} + \cdots + \kappa_{j,\lceil\vartheta\rceil} q^{i_{j,\lceil\vartheta\rceil}-1} + l_j\rfloor \\ && \mbox{ with } 0 \le l_j < q^{i_{j,\lceil\vartheta\rceil}-1} \mbox{and } 1 \le \kappa_{j,l} < q \mbox{ for } j = 1,\ldots, s\},
\end{eqnarray*}
where $\lfloor\cdot \rfloor$ just means that the contributions of $i_{j,l} < 1$ are to be ignored. Let $\mu(\bsi_{s,\lceil\vartheta\rceil}) = i_{1,1} + \cdots + i_{1,\lceil\vartheta\rceil-1} + (\vartheta - \lfloor\vartheta\rfloor) i_{1,\lceil\vartheta\rceil} + \cdots +  i_{s,1} + \cdots + i_{s,\lceil\vartheta\rceil-1} + (\vartheta - \lfloor\vartheta\rfloor) i_{s,\lceil\vartheta\rceil}$.

Then we have
\begin{eqnarray}\label{sum_Qstar}
\sum_{\bsk_\cS\in \Dcal_\cS^\ast} r_{q,\vartheta}(\bsk_\cS) & = & \sum_{i_{1,1}=1}^\infty \cdots \sum_{i_{1,\lceil\vartheta\rceil}= 1}^{i_{1,\lceil\vartheta\rceil-1}-1} \cdots \sum_{i_{s,1}=1}^\infty \cdots \sum_{i_{s,\lceil\vartheta\rceil}=1}^{i_{s,\lceil\vartheta\rceil-1}-1} \frac{| \Dcal^\ast_{S}(\bsi_{s,\lceil\vartheta\rceil})|}{q^{\mu(\bsi_{s,\lceil\vartheta\rceil})}}.
\end{eqnarray}
Some of the sums above can be empty in which case we just set the corresponding summation index $i_{j,l}=0$.

Note that by the $(t,\lceil\vartheta\rceil,\beta,n \times m,s)$-net
property we have that
$|\Dcal^\ast_{\cS}(\bsi_{s,\lceil\vartheta\rceil})| = 0$ as long as
$i_{1,1} + \cdots + i_{1,\lceil\vartheta\rceil} + \cdots + i_{s,1} +
\cdots + i_{s,\lceil\vartheta\rceil} \le \beta n - t$. Hence let now
$0\le i_{1,1}, \ldots, i_{s,\lceil\vartheta\rceil}$ be given such
that $i_{1,1},\ldots, i_{s,1}\ge 1$, $i_{j,\lceil\vartheta\rceil}<
\cdots < i_{j,1}$ for $j = 1,\ldots, s$ and where if $i_{j,l} < 1$
we set $i_{j,l}=0$ (in which case we also have $i_{j,l+1}= i_{j,l+2}
= \ldots = 0$ and the inequalities $i_{j,l} > \cdots >
i_{j,\lceil\vartheta\rceil}$ are ignored)  and  $i_{1,1} + \cdots +
i_{1,\lceil\vartheta\rceil} + \cdots + i_{s,1} + \cdots  +
i_{s,\lceil\vartheta\rceil} > \beta n - t$. We now need to estimate
$|\Dcal^\ast_{\cS}(\bsi_{s,\lceil\vartheta\rceil})|$, that is we
need to count the number of $\bsk \in \Dcal^\ast_{\cS}$ with $k_j =
\lfloor \kappa_{j,1} b^{i_{j,1}-1} + \cdots +
\kappa_{j,\lceil\vartheta\rceil}b^{i_{j,\lceil\vartheta\rceil}-1} +
l_j\rfloor$.

There are at most $(q-1)^{\lceil\vartheta\rceil s}$ choices for $\kappa_{1,1},\ldots, \kappa_{s,\lceil\vartheta\rceil}$ (we write at most because if $i_{j,l} < 1$ then the corresponding $\kappa_{j,l}$ does not have any effect and therefore need not to be included).

Let now $1\le \kappa_{1,1},\ldots, \kappa_{s,\lceil\vartheta\rceil} < q$ be given and define $$\vec{g} = \kappa_{1,1} c_{1,i_{1,1}}^\top + \cdots + \kappa_{1,\lceil\vartheta\rceil} c_{1,i_{1,\lceil\vartheta\rceil}}^\top + \cdots + \kappa_{s,1} c_{s,i_{s,1}}^\top + \cdots + \kappa_{s,\lceil\vartheta\rceil} c_{s,i_{s,\lceil\vartheta\rceil}}^\top,$$ where we set $c^\top_{j,l}=0$ if $l < 1$ or $l > n$. Further let $$B = (c_{1,1}^\top,\ldots, c_{1,i_{1,\lceil\vartheta\rceil}-1}^\top,\ldots, c_{s,1}^\top,\ldots, c_{s,i_{s,\lceil\vartheta\rceil}-1}^\top).$$ Now the task is to count the number of solutions $\vec{l}$ of $B \vec{l} = \vec{g}$.

As long as the columns of $B$ are linearly independent the number of solutions can at most be $1$. By the $(t,\lceil\vartheta\rceil,\beta,n \times m,s)$-net property this is certainly the case if (we write $(x)_+ = \max(x,0)$)
\begin{eqnarray*}
(i_{1,\lceil\vartheta\rceil}-1)_+ + \cdots + (i_{1,\lceil\vartheta\rceil}-\lceil\vartheta\rceil)_+ + \cdots &&\\ + (i_{s,\lceil\vartheta\rceil}-1)_+ + \cdots + (i_{s,\lceil\vartheta\rceil}-\lceil\vartheta\rceil)_+    &\le & \lceil\vartheta\rceil (i_{1,\lceil\vartheta\rceil} + \cdots + i_{s,\lceil\vartheta\rceil})  \\ & \le & \beta n - t,
\end{eqnarray*}
that is, as long as $$i_{1,\lceil\vartheta\rceil} + \cdots + i_{s,\lceil\vartheta\rceil} \le \frac{\beta n - t}{\lceil\vartheta\rceil}.$$

Let now $i_{1,\lceil\vartheta\rceil} + \cdots + i_{s,\lceil\vartheta\rceil} > \frac{\beta n - t}{\lceil\vartheta\rceil}$. Then by considering the rank of the matrix $B$ and the dimension of the space of solutions of $B\vec{l} = \vec{0}$ it follows the number of solutions of $B\vec{l} = \vec{g}$ is smaller or equal to $q^{i_{1,\lceil\vartheta\rceil} + \cdots + i_{s,\lceil\vartheta\rceil}-\lfloor(\beta n - t)/\lceil\vartheta\rceil\rfloor}$. Thus we have $$|\Dcal^\ast_{\cS}(\bsi_{s,\lceil\vartheta\rceil})| = 0$$ if $\sum_{j=1}^s \sum_{l = 1}^{\lceil\vartheta\rceil} i_{j,l}  \le \beta n -t$, we have $$|\Dcal^\ast_{\cS}(\bsi_{s,\lceil\vartheta\rceil})| = (q-1)^{s\lceil\vartheta\rceil}$$ if $\sum_{j=1}^s \sum_{l=1}^{\lceil\vartheta\rceil} i_{j,l}  > \beta n -t$ and $\sum_{j=1}^s i_{j,\lceil\vartheta\rceil} \le \frac{\beta n -t}{\lceil\vartheta\rceil}$ and finally we have $$ |\Dcal^\ast_{\cS}(\bsi_{s,\lceil\vartheta\rceil})| \le (q-1)^{s \lceil\vartheta\rceil} q^{i_{1,\lceil\vartheta\rceil} + \cdots + i_{s,\lceil\vartheta\rceil} - \lfloor (\beta n -t)/\lceil\vartheta\rceil\rfloor}$$ if $\sum_{j=1}^s\sum_{l=1}^{\lceil\vartheta\rceil} i_{j,l} > \beta n -t$ and $\sum_{j=1}^s i_{j,\lceil\vartheta\rceil} > \frac{\beta n -t}{\lceil\vartheta\rceil}$.

We estimate the sum (\ref{sum_Qstar}) now. Let $S_1$ be the sum in (\ref{sum_Qstar}) where $ i_{1,1} + \cdots + i_{s,\lceil\vartheta\rceil} > \beta n -t$ and $i_{1,\lceil\vartheta\rceil} + \cdots + i_{s,\lceil\vartheta\rceil} \le \frac{\beta n -t}{\lceil\vartheta\rceil}$. Let $l_1 = i_{1,1} + \cdots + i_{1,\lceil\vartheta\rceil-1} + \cdots + i_{s,1} + \cdots + i_{s,\lceil\vartheta\rceil-1}$ and let $l_2 = i_{1,\lceil\vartheta\rceil} + \cdots + i_{s,\lceil\vartheta\rceil}$. Let $A(l_1 + l_2)$ denote the number of admissible choices of $i_{1,1},\ldots,i_{s,\lceil\vartheta\rceil}$ such that $l_1 + l_2 = i_{1,1} + \cdots + i_{s,\lceil\vartheta\rceil}$. Then we have $$S_1 = (q-1)^{s \lceil\vartheta\rceil} \sum_{l_2 = 0}^{\lfloor \frac{\beta n -t}{\lceil\vartheta\rceil}\rfloor} \frac{1}{q^{(\vartheta - \lfloor \vartheta\rfloor) l_2}}  \sum_{l_1 = \beta n -t +1 - l_2}^{\infty} \frac{A(l_1+l_2)}{b^{l_1}}.$$  We have $A(l_1 + l_2) \le {l_1+ l_2 +s\lceil\vartheta\rceil -1 \choose s\lceil\vartheta\rceil - 1}$ and hence we obtain $$S_1 \le (q-1)^{s\lceil\vartheta\rceil}  \sum_{l_2 = 0}^{\lfloor \frac{\beta n -t}{\lceil\vartheta\rceil} \rfloor} \frac{1}{q^{(\vartheta - \lfloor \vartheta\rfloor) l_2}}  \sum_{l_1 = \beta n -t +1 - l_2}^{\infty} \frac{1}{q^{l_1}} {l_1+l_2 +s\lceil\vartheta\rceil - 1 \choose s \lceil\vartheta\rceil -1}.$$

From a result by Matou\v{s}ek~\cite[Lemma~2.18]{matou}, see also \cite[Lemma~6]{DP05}, we have
\begin{eqnarray*}
\lefteqn{(q-1)^{s\lceil\vartheta\rceil} \sum_{l_1 = \beta n -t +1 - l_2}^{\infty} \frac{1}{q^{l_1}} {l_1+l_2 +s\lceil\vartheta\rceil - 1 \choose s \lceil\vartheta\rceil -1} } \qquad\qquad\qquad\qquad\qquad \\ & \le & q^{l_2 - \beta n + t - 1 + s\lceil\vartheta\rceil} {\beta n - t + s\lceil\vartheta\rceil \choose s \lceil\vartheta\rceil - 1}
\end{eqnarray*}
and further we have $$ \sum_{l_2 = 0}^{\lfloor \frac{\beta n -t}{\lceil\vartheta\rceil}\rfloor} \frac{q^{l_2}}{q^{(\vartheta - \lfloor \vartheta\rfloor) l_2}} = \sum_{l_2 = 0}^{\lfloor \frac{\beta n -t}{\lceil\vartheta\rceil}\rfloor} q^{l_2(\lceil\vartheta\rceil - \vartheta)}  = \frac{q^{(\lceil\vartheta\rceil - \vartheta)(\lfloor (\beta n - t)/\lceil\vartheta\rceil \rfloor + 1)}-1}{q^{\lceil\vartheta\rceil - \vartheta}-1}.$$ Thus we obtain
\begin{eqnarray*}
S_1 &\le & \frac{q^{(\lceil\vartheta\rceil - \vartheta)(\lfloor (\beta n - t)/\lceil\vartheta\rceil \rfloor + 1)}-1}{q^{\lceil\vartheta\rceil - \vartheta}-1} q^{- \beta n + t - 1 + s\lceil\vartheta\rceil} {\beta n - t + s\lceil\vartheta\rceil \choose s \lceil\vartheta\rceil - 1} \\ & \le & \frac{q^{s\lceil\vartheta\rceil-1}}{1-q^{\vartheta- \lceil\vartheta\rceil}} {\beta n - t + s\lceil\vartheta\rceil \choose s \lceil\vartheta\rceil - 1} q^{-\vartheta\lfloor (\beta n - t)/\lceil\vartheta \rceil \rfloor}.
\end{eqnarray*}

Let $S_2$ be the part of (\ref{sum_Qstar}) for which $ i_{1,1} + \cdots + i_{s,\lceil\vartheta\rceil} > \beta n -t$ and $i_{1,\lceil\vartheta\rceil} + \cdots + i_{s,\lceil\vartheta\rceil} > \frac{\beta n -t}{\lceil\vartheta\rceil}$, i.e., we have
\begin{eqnarray*}
S_2 & \le & (q-1)^{s\lceil\vartheta\rceil} \sum_{i_{1,1}=1}^\infty \cdots \sum_{i_{1,\lceil\vartheta\rceil}= 1}^{i_{1,\lceil\vartheta\rceil-1}-1} \cdots \\ && \sum_{i_{s,1}=1}^\infty \cdots \sum_{i_{s,\lceil\vartheta\rceil}=1}^{i_{s,\lceil\vartheta\rceil-1}-1} \frac{q^{- \lfloor (\beta n - t)/\lceil\vartheta\rceil \rfloor} q^{(i_{1,\lceil\vartheta\rceil} + \cdots + i_{s,\lceil\vartheta\rceil}) (\lceil\vartheta\rceil -\vartheta)}}{q^{i_{1,1} + \cdots + i_{1,\lceil\vartheta\rceil-1} + \cdots + i_{s,1} + \cdots + i_{s,\lceil\vartheta\rceil-1}}},
\end{eqnarray*}
where we have the additional conditions $ i_{1,1} + \cdots + i_{s,\lceil\vartheta\rceil} > \beta n -t$ and $i_{1,\lceil\vartheta\rceil} + \cdots + i_{s,\lceil\vartheta\rceil} > \frac{\beta n -t}{\lceil\vartheta\rceil}$. As above let $l_1 = i_{1,1} + \cdots + i_{1,\lceil\vartheta\rceil-1} + \cdots + i_{s,1} + \cdots + i_{s,\lceil\vartheta\rceil-1}$ and let $l_2 = i_{1,\lceil\vartheta\rceil} + \cdots + i_{s,\lceil\vartheta\rceil}$. Let $A(l_1 + l_2)$ denote the number of admissible choices of $i_{1,1},\ldots,i_{s,\lceil\vartheta\rceil}$ such that $l_1 + l_2 = i_{1,1} + \cdots + i_{s,\lceil\vartheta\rceil}$. Note that $l_1 >  \lfloor \vartheta \rfloor l_2$. Then we have $A(l_1 + l_2) \le {l_1+ l_2 +s\lceil\vartheta\rceil -1 \choose s\lceil\vartheta\rceil - 1}$ and hence we obtain
\begin{eqnarray*}
S_2  & \le &  (q-1)^{s\lceil\vartheta\rceil} q^{- \lfloor (\beta n - t)/\lceil\vartheta\rceil \rfloor} \\  && \sum_{l_2 = \lfloor \frac{\beta n -t}{\lceil\vartheta\rceil} \rfloor +1}^\infty q^{(\lceil\vartheta\rceil - \vartheta) l_2}  \sum_{l_1 = \lfloor\vartheta\rfloor l_2+1}^{\infty} \frac{1}{q^{l_1}} {l_1+l_2 +s\lceil\vartheta\rceil - 1 \choose s \lceil\vartheta\rceil -1} \\ & = &  (q-1)^{s\lceil\vartheta\rceil} q^{- \lfloor (\beta n - t)/\lceil\vartheta\rceil \rfloor} \\ && \sum_{l_2 = \lfloor \frac{\beta n -t}{\lceil\vartheta\rceil} \rfloor +1}^\infty \sum_{l_1 = 0}^{\infty} q^{-l_1+l_2-1-l_2\vartheta} {l_1+l_2 + \lfloor \vartheta\rfloor l_2 -1 +s\lceil\vartheta\rceil - 1 \choose s \lceil\vartheta\rceil -1}.
\end{eqnarray*}

By using again Matou\v{s}ek~\cite[Lemma~2.18]{matou}, see also \cite[Lemma~6]{DP05}, we have
\begin{eqnarray*}
\lefteqn{ (q-1)^{s\lceil\vartheta\rceil} \sum_{l_1 = 0}^{\infty} q^{-l_1+l_2-1-l_2\vartheta} {l_1+l_2 + \lfloor \vartheta \rfloor l_2 -1 +s\lceil\vartheta\rceil - 1 \choose s \lceil\vartheta\rceil -1} } \qquad\qquad\qquad\qquad\qquad \\ & \le & q^{s\lceil\vartheta\rceil} q^{l_2(1-\vartheta)-1} {l_2 \lceil \vartheta\rceil -1 +s \lceil\vartheta\rceil - 1 \choose s \lceil\vartheta\rceil -1}
\end{eqnarray*}
and also
\begin{eqnarray*}
\lefteqn{ q^{s\lceil\vartheta\rceil - 1 - \lfloor (\beta n - t)/\lceil\vartheta\rceil \rfloor} \sum_{l_2 = \lfloor \frac{\beta n -t}{\lceil\vartheta\rceil} \rfloor +1}^\infty q^{l_2(1-\vartheta)} {l_2 \lceil \vartheta \rceil - 1 +s \lceil\vartheta\rceil - 1 \choose s \lceil\vartheta\rceil -1} } \\ & \le & q^{s\lceil\vartheta\rceil - 1 - \lfloor (\beta n - t)/\lceil\vartheta\rceil \rfloor} \sum_{l_2 = \beta n -t}^\infty q^{l_2(1-\vartheta)/\lceil\vartheta\rceil} {l_2 + \lceil\vartheta\rceil - 1 +s \lceil\vartheta\rceil - 1 \choose s \lceil\vartheta\rceil -1}  \\ & \le & q^{s\lceil\vartheta\rceil} (1-q^{(1-\vartheta)/\lceil\vartheta\rceil})^{-s\lceil\vartheta\rceil}  {\beta n - t + \lceil\vartheta\rceil-2 + s \lceil\vartheta\rceil \choose s \lceil\vartheta\rceil -1} q^{-\vartheta  (\beta n -t)/\lceil\vartheta\rceil}.
\end{eqnarray*}
Hence we have
\begin{equation*}
S_{2} \le q^{s\lceil\vartheta\rceil} (1-q^{(1-\vartheta)/\lceil\vartheta\rceil})^{-s\lceil\vartheta\rceil}  {\beta n - t + \lceil\vartheta\rceil- 2 + s \lceil\vartheta\rceil \choose s \lceil\vartheta\rceil -1} q^{-\vartheta  (\beta n -t)/\lceil\vartheta\rceil}.
\end{equation*}

Note that we have $\sum_{\bsk_\cS\in \Dcal_\cS^\ast}
r_{q,\vartheta}(\bsk_S) = S_1 + S_2$. Let $a \ge 1$ and $b\ge 0$ be
integers, then we have $${a + b \choose b} = \prod_{i=1}^b \left(1 +
\frac{a}{i}\right) \le (1+a)^b.$$ Therefore we obtain $$S_1 \le
\frac{q^{s\lceil\vartheta\rceil-1}}{1-q^{\vartheta-
\lceil\vartheta\rceil}} (\beta n - t + 2)^{s \lceil\vartheta\rceil -
1} q^{-\vartheta\lfloor (\beta n - t)/\lceil\vartheta \rceil
\rfloor}$$ and $$S_{2} \le q^{s\lceil\vartheta\rceil}
(1-q^{(1-\vartheta)/\lceil\vartheta\rceil})^{-s\lceil\vartheta\rceil}
(\beta n - t + \lceil\vartheta\rceil)^{ s \lceil\vartheta\rceil -1}
q^{-\vartheta  (\beta n -t)/\lceil\vartheta\rceil}.$$ Thus we have
$$\sum_{\bsk_\cS\in \Dcal_\cS^\ast} r_{q,\vartheta}(\bsk_\cS) \le
C_{s,q,\vartheta} (\beta n - t + \lceil\vartheta\rceil)^{s
\lceil\vartheta\rceil - 1} q^{-\vartheta\lfloor (\beta n -
t)/\lceil\vartheta \rceil \rfloor},$$ where $$C_{s,q,\vartheta} =
q^{s\lceil\vartheta\rceil} ((q-q^{\vartheta-\lfloor\vartheta\rfloor
})^{-1} +
(1-q^{(1-\vartheta)/\lceil\vartheta\rceil})^{-s\lceil\vartheta\rceil}).$$
The result follows for the case $0 < \vartheta - \lfloor \vartheta
\rfloor < 1$.

Let now $\vartheta > 1$ be an integer. Then using the same arguments as above it can be shown that $$S_1 \le  (\beta n - t + 2)^{s \vartheta} q^{-(\beta n - t)-1+s\vartheta}$$ and $$S_2 \le q^{s\vartheta} (1-q^{1/\vartheta-1})^{-s\vartheta}  (\beta n - t + \vartheta)^{ s \vartheta -1} q^{-(\beta n -t)}.$$
Thus we have $$\sum_{\bsk_\cS\in \Dcal_\cS^\ast} r_{q,\vartheta}(\bsk_\cS) \le C'_{s,q,\vartheta} (\beta n - t + \vartheta)^{s \vartheta} q^{-(\beta n - t)},$$ where $$C'_{s,q,\vartheta} = q^{s\vartheta} (q^{-1} + (1-q^{1/\vartheta-1})^{-s\vartheta}).$$ The result now follows.
\end{proof}

\begin{remark}\rm
We note that the above lemma does not hold for $\beta > 1$ in
general. Indeed, take for example $u = \{1\}$, then $\bsk_u = (k_1)$
and choose $k_1 = q^{n}$. Then the digit vector of the first $n$
digits of $q^n$ is $(0,\ldots,0)^\top$ and hence $C_1^\top \vec{k}_1
= \vec{0}$ and hence $\bsk_{(1)} \in \Dcal^\ast_{(1)}$. Thus
$$\sum_{\bsk_{(1)} \in \Dcal^\ast_{(1)}} r_{q,\vartheta}(\bsk_{(1)})
\ge q^{-n-1}$$ and hence a counterexample can be obtained for some
choices of $n,\beta, \vartheta$.

In \cite{Dick05} we did allow $\beta > 1$, but therein we had the additional assumption that the functions are periodic. In this case we were able to show that the Walsh coefficients $r_{q,\alpha}(\bsk,\bsl) = \prod_{j=1}^s r_{q,\alpha}(k_j,l_j)$ of the reproducing kernel also satisfy the additional property that $r_{q,\alpha}(q^m k_j,q^m k_j) = r_{q,\alpha}(q^m k_j) = q^{-2\alpha m}r_{q,\alpha}(k_j,k_j)$ for all $k_j, m \in \NN$, see \cite[Lemma~15]{Dick05}. Similarly, if we would also assume here that $r_{q,\vartheta}(q^n k) = q^{-\vartheta n} r_{q,\vartheta}(k)$ and $r_{q,\vartheta}(k)$ given as above if $q \not| k$, then the above counterexample would fail as then $r_{q,\vartheta}(q^n) = r_{q,\vartheta}(1 q^n) = q^{-\vartheta(n+1)} r_{q,\vartheta}(1)$.
\end{remark}

Using the above lemma we can now obtain an upper bound on the worst-case error.
\begin{theorem}\label{th_errorbounddignet}
Let $\vartheta > 1$ be a real number and $q \ge 2$ be a prime power.
The worst-case error for multivariate integration in the Walsh
space~$\EE_{s,q,\vartheta,\bsgamma}$ using a digital
$(t,\lceil\vartheta\rceil,\beta,n \times m, s)$-net over $\FF_q$,
with $0 < \beta \le 1$, as quadrature points is for non-integers
$\vartheta$ bounded by $$e(Q_{q^m,s},\EE_{s,q,\vartheta,\bsgamma})
\le  q^{-\vartheta\lfloor (\beta n - t)/\lceil\vartheta \rceil
\rfloor} \sum_{\emptyset \neq u \subseteq\cS} \gamma_u
C_{|u|,q,\vartheta} (\beta n - t + \lceil\vartheta\rceil)^{|u|
\lceil\vartheta\rceil - 1},$$ where $$C_{|u|,q,\vartheta} =
q^{|u|\lceil\vartheta\rceil}
((q-q^{\vartheta-\lfloor\vartheta\rfloor })^{-1} +
(1-q^{(1-\vartheta)/\lceil\vartheta\rceil})^{-|u|\lceil\vartheta\rceil}),$$
and if $\vartheta$ is an integer, the worst-case error is bounded by
$$e(Q_{q^m,s},\EE_{s,q,\vartheta,\bsgamma}) \le q^{-(\beta n - t)}
\sum_{\emptyset \neq u \subseteq\cS} \gamma_u C'_{|u|,q,\vartheta}
(\beta n - t + \vartheta)^{|u| \vartheta},$$ where
$$C'_{|u|,q,\vartheta} = q^{|u|\vartheta} (q^{-1} +
(1-q^{1/\vartheta-1})^{-|u|\vartheta}).$$
\end{theorem}

As a direct consequence of Corollary~\ref{cor_smoothfwalsh} we obtain the following result.
\begin{corollary}\label{cor_errorbounddignet}
Let $\delta \ge 1$ be an integer, $0 < \lambda \le 1$ and $q \ge 2$
be a prime power. Then for any function $f:[0,1)^s \rightarrow
\real$ whose partial mixed derivatives up to order $\delta$ exist it
follows that the integration error using a digital
$(t,\delta+1,\beta,n \times m, s)$-net over $\FF_q$ with $0 < \beta
\le 1$ as quadrature points is for $0 < \lambda < 1$ bounded by
\begin{eqnarray*}
|I_s(f) - Q_{q^m,s}(f)| &\le &  q^{-(\delta+\lambda)\lfloor (\beta n
- t)/(\delta+1) \rfloor}  C_{\delta,s,q,\bsgamma}
N_{\delta,\lambda,\bsgamma}(f) \\ && \sum_{\emptyset \neq u
\subseteq\cS} \gamma_u  C_{|u|,q,\delta+\lambda} (\beta n - t +
\delta+1)^{|u| (\delta+1) - 1}
\end{eqnarray*}
and for $\lambda = 1$ the integration error is bounded by
\begin{eqnarray*}
|I_s(f) - Q_{q^m,s}(f)| & \le &  q^{-(\beta n - t)}
C_{\delta,s,q,\bsgamma} N_{\delta,\lambda,\bsgamma}(f) \\
&& \sum_{\emptyset \neq u \subseteq\cS} \gamma_u C'_{|u|,q,\delta+1}
(\beta n - t + \delta+1)^{|u| \delta+1},
\end{eqnarray*}
where the constant $C_{\delta,s,q,\bsgamma}$ is given in
Corollary~\ref{cor_smoothfwalsh} and the constants
$C_{|u|,q,\delta+\lambda}$ and $C'_{|u|,q,\delta+1}$ are given in
Theorem~\ref{th_errorbounddignet}.
\end{corollary}

Explicit constructions of digital $(t,\alpha,\min(1, \alpha/d), d m \times m,s)$-nets over $\FF_q$ for all prime powers $q$, integers $\alpha, d, m, s > 1$ are given in Section~\ref{sec_talphacons}. By choosing $d = \alpha = \lceil\vartheta\rceil = \delta + 1$, by Theorem~\ref{th_errorbounddignet} and Corollary~\ref{cor_errorbounddignet} we obtain a convergence of $\Landau(q^{-\vartheta m} m^{s \lceil\vartheta\rceil + 1})$, which is optimal even for the smooth functions contained in the Walsh space~$\EE_{s,q,\vartheta,\bsgamma}$, see \cite{shar} where a lower bound for smooth periodic functions was shown.

\begin{remark}\rm
In \cite[Remark~4]{Dick05} it was noted that if $m = n$ and $\beta > \alpha$ the $t$-value must grow with $m$ and hence the restriction $\beta \le \alpha$ was added. A similar argument yields in our case that the $t$-value must grow with $n$ if $\beta n > \alpha m$ as Theorem~\ref{th_errorbounddignet} shows a convergence of $\Landau(q^{-\beta n + t})$ but the best possible convergence rate is $q^{-\alpha m}$, hence the restriction $\beta \le \alpha m/n$ was added.
\end{remark}

In case the smoothness of the function is not known our constructions adjust themselves automatically up to a certain degree in the following way: for the construction of the digital net we choose some value of $d \ge 1$ and construct a digital $(t,\alpha,\min(1,\alpha/d), dm \times m, s)$-net or a digital $(t,\alpha,\min(1,\alpha/d),d,s)$-sequence for all $\alpha \ge 1$. The values $\delta \ge 1$ and $0 < \lambda \le 1$ determine the real smoothness of the function, which we now assume is not known. The value of $\alpha$ is the smoothness analog for the digital net, i.e., we need to choose $\alpha = \delta + 1$. First assume that $\delta + \lambda \le d$, then $\min(1,\alpha/d) = (\delta+ 1)/d$ and therefore we have $\beta = (\delta+1)/d$. As $n = d m$ it follows that $\beta n = (\delta+1) m$ and therefore Corollary~\ref{cor_errorbounddignet} shows that we achieve a convergence of $\Landau(q^{-(\delta+\lambda) m} m^{s(\delta+1)+1})$, which is optimal. Now assume on the other hand that $\delta + \lambda > d$, then $\min(1,\alpha/d) = 1$ and therefore $\beta = 1$. Again we have $n = d m$ and hence $\beta n = d m$. In this case Corollary~\ref{cor_errorbounddignet} shows that our construction achieves a convergence of $\Landau(q^{-d m} m^{s(\delta+1)+1})$.

Note that numerical integration of functions with less smoothness, i.e., for example functions with partial mixed derivatives up to degree 1 in $\LL_{2}([0,1)^s)$ or functions with bounded variation, has been considered in many papers and monographs, see for example \cite{DKPS,DP05,DSWW1,DSWW2,KN,niesiam,SW98,sob67}. Using the notation from above, basically those results are concerned with the case where $\delta = 0$ and $\lambda = 1$, hence the results here are a direct continuation of what was previously known. The construction of digital nets proposed here for $d = 1$ yields obviously digital $(t,m,s)$-nets and $(t,s)$-sequences as for example defined in \cite{niesiam}. In view of Corollary~\ref{cor_errorbounddignet} and the explanation which followed it is hence not surprising that the classical examples and theory (see for example \cite{DKPS,DP05,hlawka,koksma,KN,niesiam,PDP,sob67}) only yielded a convergence of $\Landau(q^{m(-1+\varepsilon)})$ for any $\varepsilon > 0$ (the $\varepsilon$ here is used to hide the powers of $m$).

Note that the worst-case error in the Walsh
space~$\EE_{s,q,\vartheta,\bsgamma}$ is invariant with respect to a
digital shift (see \cite{DP05}), hence
Corollary~\ref{cor_errorbounddignet} also holds for digitally
shifted digital nets. Thus, if one wants to use randomized digital
nets, one can also use randomly digitally shifted digital nets. The
root mean square worst-case error for this case would of course be
bounded by the bound in Corollary~\ref{cor_errorbounddignet}, as
this bound holds for any digital shift, i.e., our result here is
even stronger in that we have shown that even for the worst digital
shift we still have the bound of
Corollary~\ref{cor_errorbounddignet}. From this, it follows that for
our situation here, there is, in some sense, no bad digital shift.
Other more sophisticated scrambling methods which do not destroy the
essential properties of the point set can be used as well (for
example a digital shift of depth $m$, see \cite{DP05b,matou}), see
\cite{owenscr} for some ideas in this direction.

\end{document}